\documentclass[10pt,leqno]{amsart}
\usepackage{amsmath,amsfonts,amssymb,amsthm} 
\usepackage[latin1]{inputenc}
\usepackage[T1]{fontenc}
\usepackage{bbm}
\usepackage{amsmath}
\usepackage{amsthm}
\usepackage{latexsym}
\usepackage{amssymb}
\usepackage[all]{xy}
\usepackage{calc}
\usepackage{multicol}
\usepackage{bbm}

\usepackage{todonotes}

\newtheorem{thm}{Theorem}[section]
\newtheorem{prop}{Proposition}[section]

\newtheorem{lem}{Lemma}[section]
\newtheorem{cor}{Corollary}[section]

\newtheorem{rmq}{Remark}[section]

\numberwithin{equation}{section}

\newcounter{exercice}

\DeclareMathOperator{\im}{Im}

\DeclareMathOperator{\supp}{supp}

\newcommand{\field}[1]{\mathbb{#1}}
\newcommand{\N}{\field{N}}          		
\newcommand{\Z}{\field{Z}}          		
\newcommand{\R}{\field{R}}          		
\newcommand{\Sph}{\field{S}}        		
\newcommand{\C}{\field{C}}          		
\newcommand{\Hy}{\field{H}}

\DeclareMathOperator\sgn{sgn}

\newcommand{\one}{\mathbbm{1}}

\begin{document}
\title[4NLS]{Fourth order Schr\" odinger equation with mixed dispersion on certain Cartan-Hadamard manifolds}
\author[Jean-Baptiste Casteras]{Jean-Baptiste Casteras}
\address{CMAFCIO, Faculdade de Ci\^encias da Universidade de Lisboa, Edificio C6, Piso 1, Campo Grande 1749-016 Lisboa, Portugal}
\email{jeanbaptiste.casteras@gmail.com}
\author[Ilkka Holopainen]{Ilkka Holopainen}
\address{Department of Mathematics and Statistics, P.O. Box 68, 00014 University of
Helsinki, Finland}
\email{ilkka.holopainen@helsinki.fi}
\thanks{J.-B.C. supported by FCT - Funda\c c\~ao para a Ci\^encia e a Tecnologia, under the project: UIDB/04561/2020; I.H. supported by the Magnus Ehrnrooth foundation}

\begin{abstract}
This paper is devoted to the study of the following fourth order Schr\" odinger equation with mixed dispersion on $M^N$, an $N$-dimensional Cartan-Hadamard manifold. Namely we consider
\begin{equation}
\label{abstracteq}\tag{4NLS}
\begin{cases}
i\partial_t\psi = -\Delta_M^2\psi +\beta\Delta_M \psi +\lambda|\psi|^{2\sigma}\psi \quad\text{in }\R\times M,\\
\psi(0,\cdot)=\psi_0\in X,
\end{cases}
\end{equation}
where $\beta \geq 0$, $\lambda=\{-1,1\}$, $0<\sigma < 
4/(N-4)_+$, $\Delta_M$ is the Laplace-Beltrami operator on $M$ and $X=L^2 (M)$ or $X=H^2 (M)$. At first,  we focus on the case where $M$ is the hyperbolic space $\Hy^N$. Using the fact that there exists a Fourier transform on this space, we prove the existence of a global solution to \eqref{abstracteq} as well as scattering for small initial data provided that $N\ge 4$ and $0<\sigma < 4/N$ if $X=L^2 (\Hy^N)$ or $0<\sigma < 4/(N-4)_+$ if $X=H^2 (\Hy^N)$. Next, we obtained weighted Strichartz estimates for radial solutions to \eqref{abstracteq} on a large class of rotationally symmetric manifolds by adapting the method of Banica and Duyckaerts (Dyn. Partial Differ. Equ., 07).
Finally,
we give a blow-up result for a rotationally symmetric manifold relying on a localized virial argument.  
\end{abstract}

\maketitle
\section{Introduction}
We consider the following biharmonic (i.e. fourth-order) nonlinear Schr\"odinger equation with mixed dispersion on $M^N$, an $N$-dimensional Cartan-Hadamard manifold,
\begin{equation}\label{4NLS}
\begin{cases}
i\partial_t\psi =- \Delta_M^2 \psi +\beta\Delta_M \psi -\lambda |\psi|^{2\sigma}\psi \quad\text{in }\R\times M,\\
\psi(0,\cdot)=\psi_0\in X,
\end{cases}
\end{equation}
where $\beta \geq 0$, 
\[
0<\sigma <\dfrac{4}{(N-4)_+}=\begin{cases}\dfrac{4}{N-4},&\mbox{if }\  N>4,\\ \infty,&\mbox{if }\   N\leq 4, \end{cases},
\] 
 $\Delta_M$ is the Laplace-Beltrami operator on $M$ and $X=L^2 (M)$ or $X=H^2 (M)$. When $M=\R^N$, Karpman and Shagalov \cite{MR1779828} introduced the fourth order term to regularize and stabilize the classical Schr\" odinger equation. We refer to \cite{MR1898529,MR4001029,MR3855391} for more details concerning the stability of standing wave solutions to this equation. 
 Our main focus in this paper will be to obtain well-posedness results for \eqref{4NLS} on Cartan-Hadamard manifolds. Before proceeding, let us recall some results in the Euclidean space. The global existence of solution to \eqref{4NLS} was proved by Pausader \cite{MR2505703,MR2502523}. To do so, he used the dispersive estimates of Ben-Artzi, Koch and Saut \cite{MR1745182}. More precisely, let 
$$I_\beta (t,x)= \int_{\R^N} e^{it (|\xi|^4 +\beta |\xi|^2) -i \langle x ,\xi\rangle} d\xi, $$
be the propagator of our operator in $\R^N$. Then there holds:
\begin{itemize}
\item If $\beta=0$,
$$|D^\alpha I_0 (t,x)|\leq C t^{-\frac{N+|\alpha|}{4}} \big(1+ t^{-\frac{1}{4}}|x| \big)^{\frac{|\alpha|-N}{3}},$$
for all $t>0$ and $x\in \R^N$.
\item For all $0<t\leq 1$ and all $x\in \R^N$ or for all $t>0$ and all $|x|\geq t$,
$$|D^\alpha I_\beta (t,x)| \leq  C t^{-\frac{N+|\alpha|}{4}} \big(1+ t^{-\frac{1}{4}}|x| \big)^{\frac{|\alpha|-N}{3}} , $$
\item For $\beta>0$,  all $t\geq 1$ and $|x|\leq t$,
$$|D^\alpha I_{\beta} (t,x)| \leq C t^{-\frac{N+|\alpha|}{2}} \big(1+ t^{-\frac{1}{2}}|x| \big)^{|\alpha|} .$$
\end{itemize}

Thanks to these estimates, Pausader was able to obtain the following Strichartz estimates. Before stating them, let us introduce some notation. We say that a pair $(q,r)$ is S-admissible if $2\leq q,r\leq \infty$, $(q,r,N)\neq (2,\infty , 2)$, and
$$\dfrac{2}{q}+ \dfrac{N}{r}=\dfrac{N}{2}.$$
We say that a pair $(q,r)$ is B-admissible if $2\leq q,r\leq \infty$, $(q,r,N)\neq (2,\infty , 4)$, and
$$\dfrac{4}{q}+ \dfrac{N}{r}=\dfrac{N}{2}.$$
Let $I=[0,T]$ and $\psi\in C(I,H^{4})$ be a solution to 
\begin{equation}
\label{4nlsreal}
\begin{cases}
i\partial_t\psi =- \Delta_{\R^N}^2\psi +\beta\Delta_{\R^N}\psi -h\quad\text{in }I\times \R^N,\\
\psi(0,\cdot)=\psi_0\in L^2 (\R^N),
\end{cases}
\end{equation}
for some suitable function $h$ smooth enough. Assume in addition that $T\leq 1$ if $\beta <0$. Then, for any $B$-admissible pairs $(q,r)$ and $(\tilde{q} ,\tilde{r})$, we have
\begin{equation}
\label{strireal}
\|\psi\|_{L^q (I,L^r)} \leq C \big(\|\psi_0\|_{L^2} + \|h\|_{L^{\tilde{q}^\prime } (I,L^{\tilde{r}^\prime })}\big),
\end{equation}
whenever the right hand side is finite, where $\tilde{q}^\prime$ and $\tilde{r}^\prime$ are the conjugate exponents of $\tilde{q}$ and $\tilde{r}$. We also have, for $S$-admissible pairs $(q,r)$, $(\tilde{q} ,\tilde{r})$, and any $s\geq 0$,
$$\||\nabla|^s \psi\|_{L^q (I,L^r)} \leq C \big(\||\nabla |^{s-\frac{2}{q}}\psi_0\|_{L^2} + \||\nabla|^{s-\frac{2}{q} - \frac{2}{\tilde{q}}}h\|_{L^{\tilde{q}^\prime } (I,L^{\tilde{r}^\prime })}\big),$$
whenever the right side is finite. Now, let us take $h=\lambda |\psi|^{2\sigma}\psi$ and $\psi_0 \in \R^N$ in \eqref{4nlsreal}, with $\lambda=\{-1,1\}$ and $\sigma <4/(N-4)_+ $. 
 Then, using the conservation of the mass
\[ M[\psi (t)]:=\int_{\R^N}|\psi (t)|^2 \,dx=M[\psi_0],
\]
and the conservation of the energy
\[
E[\psi (t)] :=\frac{1}{2}\int_{\R^N}|\Delta_{\R^N} \psi (t)|^2\, dx + \frac{\beta}{2}\int_{\R^N}|\nabla\psi (t)|^2\, dx
-\frac{\lambda}{2\sigma +2}\int_{\R^N}|\psi (t)|^{2\sigma+2}dt= E[\psi_0],
\]
Pausader \cite{MR2505703} showed that there exists a global solution in the following cases:
\begin{itemize}
\item $\lambda \geq 0$;
\item $\lambda <0$ and $\sigma< 4/N$;
\item $\lambda <0$, $\sigma=4/N$ and $\psi_0$ is sufficiently small in $L^2$;
\item $\lambda <0$, $\sigma \leq 4/(N-4) $ and $\psi_0$ is sufficiently small in $H^2$.
\end{itemize}
He also proved scattering in $H^2 (\R^N)$ provided that $\lambda \geq 0$, $N\geq 5$, $\beta \geq 0$ and $4/N < \sigma < 4/(N-4)$. Let us also point out that Pausader also obtained some results when $\sigma= 4/(N-4)$ but we will not consider this case.

In view of these results, it seems quite natural to investigate in which spaces the Strichartz estimates hold. For the classical Schr\" odinger equation i.e. for $i\partial_t \psi =-\Delta_M \psi +h $, $\psi(0)=\psi_0\in L^2 (M)$, where $M$ is some manifold, this question has been intensively studied in recent decades. It has been proved that the Strichartz estimates holds for a much larger class of pairs $(q,r)$, $(\tilde{q} ,\tilde{r})$, in the hyperbolic space or more generally in Damek-Ricci spaces than in the Euclidean one see \cite{MR2765426,MR2566713,banica2007,BaCaSt,MR2245889}. More precisely, for any S-admissible pair $(q,r)$ and $(\tilde{q} ,\tilde{r})$, we have
$$\|\psi\|_{L^q (I,L^r (\R^N))} \leq C \big(\|\psi_0\|_{L^2 (\R^N)} + \|h\|_{L^{\tilde{q}^\prime } (I,L^{\tilde{r}^\prime } (\R^N))}\big),$$
whereas in $\Hy^N$, the same inequality holds for pairs $(p,q)$ and $(\tilde{p} ,\tilde{q})$ belonging to the triangle
$$\left\lbrace \left(\dfrac{1}{p},\dfrac{1}{q}\right) \in (0,\dfrac{1}{2}] \times (0,\dfrac{1}{2} ) \colon \dfrac{2}{p}+ \dfrac{N}{q}\geq\dfrac{N}{2}\right\rbrace \cup \left\lbrace\left(0,\dfrac{1}{2}\right)\right\rbrace.$$
Intuitively, on a non-compact negatively curved manifold, one can expect to have better Strichartz estimates than in the Euclidean space because of a faster volume growth.
On the other hand, on a compact manifold $M$, Strichartz estimates only hold with loss of derivatives \cite{MR2058384} i.e. if $h=0$ and $I$ is bounded, then we have, for any S-admissible pair $(p,q)$,
$$\|\psi\|_{L^p (I,L^q (M))} \leq C \|\psi_0\|_{H^{1/p} (M)}. $$
In some compact manifolds, this inequality is known to be sharp \cite{MR1909648}. 

In the present paper, we will show that the Strichartz estimates \eqref{strireal} can be improved on the hyperbolic space and on a class of rotationally symmetric Cartan-Hadamard manifolds. Our first main result concerns the case $M=\Hy^N$. In this case, we say that $(p,q)$ is an admissible pair if 
$$\left(\frac{1}{p},\frac{1}{q}\right)\in  \left\lbrace\left(0,\frac{1}{2} \right)\times \left(0,\frac{1}{2}\right) \colon \frac{4}{p} +\frac{N}{q} \geq \frac{N}{2} \right\rbrace \cup \left\lbrace\left(0,\frac{1}{2}\right) \right\rbrace.$$

Also, since the bottom of the continuous spectrum of the Laplace-Beltrami operator on $\Hy^N$ is $\rho^2 = (N-1)^2/4$, instead of considering the operator $\Delta_{\Hy^N}^2 - \beta \Delta_{\Hy^N}$, we set
$$P=(\Delta_{\Hy^N} +\rho^2)^2 - \beta (\Delta_{\Hy^N} +\rho^2).$$
By definition, the bottom of the continuous spectrum of $P$ is $0$.  

\begin{thm}
\label{introstri}
Suppose that $N\geq 2$. Let $\psi$ be a solution to
$$\begin{cases} i\partial_t \psi +P \psi=h\mbox{ in}\ \R\times \Hy^N\\ \psi(0,\cdot)=\psi_0  \in L^2 (\Hy^N ),\end{cases}$$
on $I=[0,T]$. Assume that $|I|\leq 1$ if $\beta \leq 0$. Let $(p,q)$ and $(\tilde{p},\tilde{q})$ be admissible pairs. Then there exists a constant $C>0$ independent of  $T$ such that
\begin{equation*}
\|\psi\|_{L_t^p (I,L_x^q (\Hy^N))}\leq C \big(\|\psi_0\|_{L^2 (\Hy^N)} +\|h\|_{L_t^{\tilde{p}^\prime }(I, L_x^{\tilde{q}^\prime} (\Hy^N))}\big).
\end{equation*}
\end{thm}

Thanks to this Strichartz estimate, we can easily prove some global well-posedness and scattering results for
\begin{equation}
\label{4NLSeqintro}
\begin{cases}i\partial_t \psi +P \psi =\lambda|\psi|^{2\sigma}\psi \mbox{ in}\ \R\times \Hy^N\\ \psi(0,\cdot)=\psi_0.\end{cases}
\end{equation}

\begin{thm}
\label{thmexintro}
Let $N\geq 2$, $\lambda=\{-1,1\}$, $X_0=L^2 (\Hy^N)$ and $X_1 =H^2 (\Hy^N)$. For $i=0,1$, if $\sigma < \tilde{\sigma}_i= \begin{cases}  \tfrac{4}{N},& i=0\\ \tfrac{4}{(N-4)_+} ,&i=1 \end{cases}$, \eqref{4NLSeqintro} is locally well-posed for arbitrary $X_i$ data, i.e. for any $\psi_0\in X_i (\mathbb{H}^N)$, there exists $T>0$ such that the solution $\psi(t,\cdot)$  to \eqref{4NLSeqintro} exists for all $|t|< T$. Moreover, if $\sigma \leq \tilde{\sigma}_i$ and $\beta>0$, then there exists a constant $c>0$ such that if $\|\psi_0\|_{X_i (\mathbb{H}^N )}\leq c$ then \eqref{4NLSeqintro} is globally well-posed, i.e. $\psi (t,\cdot)$ exists for all $t\geq 0$. We also have that $\psi$ satisfies the following scattering property: there exist $\psi_\pm \in L^2 (\Hy^N)$ such that
$$\|\psi(t)- e^{itP}\psi_\pm \|_{X_i (\Hy^N)} \rightarrow 0\ as\ t\rightarrow \pm \infty.$$
\end{thm}
\begin{rmq}
We point out that the previous results also hold for the unshifted operator if we replace the conditions on $\beta$ by conditions on $\beta - 2\rho^2$. For instance, Theorem \ref{introstri} holds true for $\beta - 2\rho^2 \leq 0$ if $|I|\leq 1$. Also notice that to get the global well-posedness result in Theorem \ref{thmexintro}, we need to assume that $\beta>0$ (or $\beta - 2\rho^2>0$ for the unshifted operator). This comes from the large time behavior of the free propagator which is not integrable. 
\end{rmq}
Notice, in particular, that our scattering result is a lot better than the one in the Euclidean space. We point out that this result holds for not necessarily radially symmetric functions. The proof of Theorem \ref{thmexintro} is very standard once the Strichartz estimate holds. It relies on a fixed point argument for the existence part for an appropriate choice of functional spaces. 

Let us make some comments on the proof of Theorem \ref{introstri}. We follow the $TT^\ast$ method of Kato \cite{Kato} and Ginibre and Velo \cite{MR1151250} (see also Keel and Tao \cite{keel-tao}). The main difficulty is to obtain $L^p - L^q$-estimates on the dispersive propagator $e^{itP}$. Denoting by $v$ the kernel of this propagator, we will deduce this estimate from a pointwise bound of $v$ and the Kunze-Stein inequality (which is an improved Young convolution inequality). A crucial ingredient in our proof is the fact that it is possible to define a Fourier transform in the hyperbolic space enjoying most of the properties of the Euclidean Fourier transform. Using this transformation, we see that
$$v(t,x)=\int_{\R}  e^{-it (\lambda^4 +\beta \lambda^2 ) } \Phi_\lambda (r) |c(\lambda)|^{-2}\, d\lambda,$$
where $c(\lambda)$ is the Harish-Chandra coefficient and $\Phi_\lambda$ is a explicit function (see \eqref{defPhi} for a precise definition). At this point, we adopt the approach of Ionescu and Staffilani \cite{IS2009} to estimate this integral. Namely we split the integral into two parts depending on whether $|\lambda|\in [0,2^J]$ or  $|\lambda|\in [2^J,\infty]$ for a suitable choice of $J$ which depends on the critical value of the phase and then apply a dyadic decomposition of $[2^J ,\infty]$. Due to the mixed dispersion of our operator, the critical point of the phase will behave differently depending if $r/|t|$ is large or not. Let us also mention that this method only gives a decay in $|t|^{-1}$ when $|t|$ is large. The same holds for the classical Schr\" odinger equation but was proved to be not optimal by Anker and Pierfelice \cite{MR2566713} who obtained a decay in $|t|^{-3/2}$. Here, we refined a bit the estimate for $|t|$ large using the recent method of \cite{AMPVZ} to get the probably optimal decay in $|t|^{-3/2}$.   
Let us also comment a bit on the method used by Anker and Pierfelice \cite{MR2566713} to obtain the pointwise decay of the classical Schr\"odinger propagator. They use the fact that the inverse Fourier transform in the hyperbolic space can be obtained as the composition of the inverse Abel transform with the inverse Fourier transform on the real line. However, to use this fact, one has to obtain very precise estimates on the derivative of the real inverse Fourier transform of $e^{-it (\cdot^4 +\beta \cdot^2 ) } $. Compared to the Euclidean case, the propagator estimates we obtained are decaying exponentially fast in space, however they are worse for $|t|$ large.
 We end this first part dealing with the hyperbolic space by proving a trapping result saying that if the initial data is large enough then the solution to \eqref{4NLSeqintro} with $\lambda=1$ does not scatter whereas if it is small enough (in a quantified way), we get global existence. 

The second part of this paper will focus on the case where $M$ is a rotationally symmetric Cartan-Hadamard manifold. We recall that a Cartan-Hadamard manifold is a complete, connected and simply connected Riemannian manifold of non-positive sectional curvature. Due to the lack of Fourier transform in this space, we adopt a very different approach to prove Strichartz estimate. Following \cite{BaDyweight}, one can see that up to a change of variables, \eqref{4NLS} can be seen as an equation in the flat Euclidean space but with potentials for radially symmetric functions. We then apply Strichartz estimates for the pure fourth order Euclidean Schr\" odinger equation and consider the terms involving potential as inhomogeneous terms. Using the method of \cite{BaDyweight}, we use a smoothing estimate to deal with those terms. The smoothing estimate is a consequence of a uniform weighted estimate for the resolvent $[(-\Delta_{\R^N} - V)(-\Delta_{\R^N} -V +\beta)]^{-1}$, for some appropriate potential $V$. Here, we rely on the fact that the resolvent of this operator can be written as the difference of resolvents of second order operator for which estimates were already known see \cite{BPST}. Our result is the following:

\begin{thm}
\label{thmexCH}
Let $M=(\R^N ,g)$ be a complete rotationally symmetric manifold with Riemannian metric $g= dr^2 + \phi^2 (r) d\theta^2$. Set $\sigma(r)= \big(\frac{r}{\phi (r)}\big)^{(N-1)/2}$ and $V=\frac{\sigma''}{\sigma}+(N-1)\frac{\phi'}{\phi}\frac{\sigma'}{\sigma} $.
Assume that $V\in C^1 (\R^n \setminus \{0\})$ such that
$$|V(x)|\leq C |x|^{-2},$$
and
\begin{equation*}
\left|\Delta_{\R^N} \frac{V}{\langle x\rangle}\right| +\left|\nabla\frac{V'}{\langle x\rangle}\right|+\dfrac{ \big| \Delta_{\R^N} V +V^2 \big|}{\langle x \rangle} \leq \frac{C}{\langle x \rangle},
\end{equation*}
where $\langle x \rangle =(1+|x|^2)^{1/2}$. In addition, suppose that there exists $\delta_0$ such that
\[
\left(\frac{N}{2} -1\right)^2 +r^2 V \geq \delta_0,
\]
\[
\left(\frac{N}{2} -1\right)^2 - r^2 \partial_r (rV) \geq \delta_0,
\]
in $\R^N \setminus \{0\}$, and
\begin{equation}
\label{mainCHass}
\|\langle x\rangle^2 (\Delta V+V^2) \|_{L^{N/2,\infty} (\R^N)},\  \|\langle x\rangle^2 V^\prime\|_{L^{N/2,\infty} (\R^N)},\ \|\langle x\rangle^2 V\|_{L^{N/2,\infty} (\R^N)}\leq C,
\end{equation}
resp.
\begin{equation}
\label{mainCHass2}
\|\langle x\rangle (\Delta V+V^2) \|_{L^{N,\infty} (\R^N)},\  \|\langle x\rangle V^\prime\|_{L^{N,\infty} (\R^N)},\ \|\langle x\rangle V\|_{L^{N,\infty} (\R^N)}\leq C.
\end{equation}
 Let $\psi$ be a radial solution to
$$\begin{cases}i\partial_t \psi-\Delta_M^2 \psi +\beta \Delta_M \psi= h\mbox{ in}\ \R \times M,\\ \psi(0,\cdot)=\psi_0.\end{cases}$$
Then, there exists $C>0$ such that, for all intervals $I$, and all $B$-admissible resp. $S$-admissible pairs $(p_i ,q_i)$, we have
\begin{align*}
&\big\|\psi \sigma^{-(1-2/q_1)}\big\|_{L^{p_1} (I, L^{q_1} (M)) }\\
&\quad \leq C \big(\|\nabla_{M} \psi_0\|_{L^2 (M)} +\|\psi_0 |\nabla_{M} \sigma|/\sigma \|_{L^2 (M)} + \|h \sigma^{(1- 2/q_2)} \|_{L^{p_2^\prime} (I, L^{q_2^\prime } (M))}  \big).
\end{align*}
\end{thm}

Let us observe that \eqref{mainCHass} (resp. \eqref{mainCHass2}) holds true if, for large $r$, $V(r)\approx r^{-m}$, with $m\geq 4$ (resp. $m\geq 2$).
As in \cite{BaDyweight}, one can show that Theorem~\ref{thmexCH} holds true for $S$-admissible pairs if for instance $\phi (r)=r+a_1 r^3 +\ldots +a_k r^{2k +1}$, where $k\geq 1$ and $a_i>0$, $i=1,\ldots ,k$, or more generally, if $\phi (r)= Ar^m +o(r^m)$, with $A>0$, $m> (N-1)^{-1}$ and the radial sectional curvature of $M$ (i.e. $-\phi^{\prime \prime}/\phi$) at a point $x$ is negative and strictly greater than $- \dfrac{1}{2(N-1) }\dfrac{1}{r^2}$, where $r(x)=d(o,x)$ is the distance to the pole $o\in M$. Concerning $B$-admissible pairs, we need to impose stronger restrictions on the manifold for our result to hold. We see that the condition $|V(r)|\leq r^{-4}$ roughly implies that $|\phi(r)| \leq r e^{C/r^2}$, for some positive constant $C$. So our theorem holds at most for manifolds whose radial sectional curvature is greater than $-1/r^4$, for $r$ large. 
Let us observe in view of the result obtained in this paper that the validity of Strichartz inequality is still unknown for rotationally symmetric manifolds where $\phi$ has an intermediate growth between polynomial and exponential (or in terms of sectional curvature $K_M$, for $-C_1\leq K_M \leq -C_2 /r^{-2}$). This question is also open for the classical Schr\" odinger equation. We refer to \cite{MR2211154,MR2296623} for some examples of $V$ for which the Strichartz inequality (related to the classical Schr\" odinger equation) does not hold. Let us point out that along our proof, we also obtain Strichartz estimate for fourth order Schr\" odinger equation with potentials which can be of independent interest. We refer to \cite{MR4069234} for related results.

We finally end this paper by computing a localized virial inequality for radial functions on generic complete rotationally symmetric manifold following Boulenger and Lenzmann \cite{bou-lenz}. Thanks to this, we provide some conditions in order to construct finite time blowing-up solutions; see the beginning of Section \ref{secblowup}. Let us point out that these conditions are hard to check due to the fact that contrarily to the flat Euclidean flat case it is not possible to use a scaling argument to get rid of what is happening in the ball of radius $R$, for $R$ large enough. To check these conditions, we refer the readers to \cite{maple} where a Maple file using the same notation as below is available to verify if the fourth point is valid or not. The first condition corresponds to "star" and the second one to "star2".

The plan of this paper is the following: in Section~\ref{secprel}, we recall some well-known facts about the hyperbolic space. We begin with some geometric facts and then recall the definition and main properties of the Fourier transform. In Section~\ref{secpropa}, we prove the pointwise estimate for the dispersive propagator. Theorems \ref{introstri} and \ref{thmexintro} are establsihed in Section~\ref{secstri} as well as our trapping result. Section~\ref{secstri2} is dedicated to the proof of Theorem \ref{thmexCH}. In the last section, we give a proof of our finite-time blow-up result. 

\subsubsection*{Data availability statement:}
My manuscript has no associated data.

\subsubsection*{Conflict of interest:}
  Authors state no conflict of  interest.

\section{Preliminaries}\label{secprel}
\subsection{The hyperbolic space}
We consider the upper (sheet) hyperboloid model for the hyperbolic space $\Hy^N$. Thus we equip $\R^{N+1}$ with 
the Minkowski metric 
\[
-(dx^0)^2+\sum_{i=1}^N (dx^i)^2
\]
and the symmetric bilinear form
\[
[x,y]= x^0y^0 -\sum_{i=1}^N
x^iy^i.
\]
The upper sheet model for the hyperbolic space is then 
\[
\Hy^N=\{x\in\R^{N+1}\colon [x,x]=1,\ x^0>0\}
\]
equipped with the Riemannian metric induced by the Minkowski metric. By fixing a pole $o=(1,0,\ldots,0)\in\Hy^N$,  $\Hy^N$ can be identified with the homogeneous space $SO^{+}(N,1)/SO(N)$, where $SO^{+}(N,1)$ is the connected Lie group of $(N+1)\times (N+1)$-matrices $X$, with $\det X=1$ and $X_{00}>0$, that keep the form $[\cdot,\cdot]$ invariant and $SO(N)$ is the subgroup of $SO^{+}(N,1)$ that fixes the pole $o$. 
Furthermore, $\Hy^N$ can be expressed as
\[
\Hy^N=\{(t,\xi)\in\R^{N+1}\colon t=\cosh r,\ \xi=\sinh r\,\omega,\ r\ge 0,\omega\in\Sph^{N-1}\},
\]
where $\Sph^{N-1}\subset o^{\perp}=\{(0,y^1,\ldots,y^N)\in \R^{N+1}\}$, the orthogonal complement of $o$ with respect to $[\cdot,\cdot]$. In (global) coordinates $(r,\omega)$, 
the $SO(N,1)$-invariant Riemannian metric of $\Hy^N$ is given by  
\[
ds^2 = dr^2 + \sinh^2 r\, d\omega^2,
\]
where $d\omega^2$ is the standard metric on $\Sph^{N-1}$, and the Riemannian volume form by 
\[
d\mu=\sinh^{N-1}r\,dr\,d\omega.
\]
The distance between $x=(\cosh r,\sinh r\,\omega)$ and the pole $o$ is $d(x,o)=r$ and, in general,
\[
d(x,y)=\cosh^{-1}[x,y]
\]
for any $x,y\in\Hy^N$. 
Finally, the Laplace-Beltrami operator on $\Hy^N$ is given by
\[
\Delta_{\Hy^N}=\partial_r^2 +(N-1)\frac{\cosh r}{\sinh r}\partial_r
+\frac{1}{\sinh^2 r}\Delta_{\Sph^{N-1}},
\]
where $\Delta_{\Sph^{N-1}}$ is the Laplace-Beltrami operator on the unit sphere $\Sph^{N-1}$.
\subsection{The Fourier transform}
Following Helgason \cite{helga}, Banica \cite{banica2007}, and Ionescu and Staffilani \cite{IS2009} we define the Fourier transform on $\Hy^N$ as follows. For $\omega\in\Sph^{N-1}$ and $\lambda\in\R$, let $b(\omega)=(1,\omega)\in\R^{N+1}$ and $h_{\lambda,\omega}\colon\Hy^N\to\C$,
\[
h_{\lambda,\omega}(x)=[x,b(\omega)]^{i\lambda-\rho},
\]
where $\rho=(N-1)/2$. The functions $h_{\lambda,\omega}$ are generalized eigenfunctions of $\Delta_{\Hy^N}$. Indeed, we have
\[
-\Delta_{\Hy^N} h_{\lambda,\omega} = \left(\lambda^2 + \rho^2\right)h_{\lambda,\omega}.
\]
The Fourier transform of a function $f\in C_0(\Hy^N)$ (the space of continuous functions in $\Hy^N$ with compact support) is now defined by
\[
\hat{f}(\lambda,\omega)=\int_{\Hy^N}f(x)h_{\lambda,\omega}(x)\,d\mu,
\]
and the Fourier inversion formula is given by
\begin{align*}
f(x)&=\int_{\R}\int_{\Sph^{N-1}}\hat{f}(\lambda,\omega)\bar{h}_{\lambda,\omega}(x)
|c(\lambda)|^{-2}\,d\lambda\,d\omega\\
&=\int_{\R}\int_{\Sph^{N-1}}\hat{f}(\lambda,\omega)[x,b(\omega)]^{-i\lambda-\rho}|c(\lambda)|^{-2}\,d\lambda\,d\omega,
\end{align*}
where $c(\lambda)$ is the Harish-Chandra coefficient
\[
|c(\lambda)|^{-2}=\frac{1}{2(2\pi)^d}\frac{|\Gamma(i\lambda+\rho)|^2}{|\Gamma(i\lambda)|^2},
\]
and $d\omega$ is the standard measure on 
$\Sph^{N-1}$ with a suitable normalisation. It follows from the definition that 
\begin{equation}\label{cl-2}
|c(\lambda)|^{-2}=c(\lambda)^{-1}c(-\lambda)^{-1};
\end{equation}
 see \cite[Prop. A1]{I-JFA2000}. A version of the  Plancherel theorem holds in $\Hy^N$. Indeed, the Fourier transform $f\mapsto\hat{f}$ extends to an isometry of $L^2(\Hy^N)$ onto $L^2\big(\R\times\Sph^{N-1},|c(\lambda)|^{-2}\,d\lambda d\omega\big)$.  
If $f\in C_0(\Hy^N)$ is $SO(N)$-invariant, i.e. radial with respect to $o$, we have
\[
\hat{f}(\lambda,\omega)=\hat{f}(\lambda)
=\int_{\Hy^N}f(x)\Phi_{-\lambda}(x)\,d\mu,
\]
and
\[
f(x)=\int_{\R} \hat{f}(\lambda)\Phi_\lambda (x)|c(\lambda)|^{-2}\,d\lambda,
\]
where 
\begin{equation}
\label{defPhi}
\Phi_\lambda (x)=\int_{\Sph^{N-1}}[x,b(\omega)]^{-i\lambda-\rho}\,d\omega.
\end{equation}
By denoting $x=(\cosh r,\sinh r\,\vartheta)$, where $r=r(x)=d(x,o)$, we can express $\Phi_\lambda$ as
\begin{align*}
\Phi_\lambda(r)&=\int_{\Sph^{N-1}} \left(\cosh r-\sinh r\,(\vartheta\cdot\omega)\right)^{-i\lambda-\rho}\,d\omega\\
&=C\int_0^\pi  \left(\cosh r-\sinh r\,\cos\theta\right)^{-i\lambda-\rho}(\sin\theta)^{N-2}\,d\theta.
\end{align*}

\section{Estimates for Schr\"odinger propagator}\label{secpropa}

The most part of this section is devoted to the pointwise estimate for the dispersive propagator. As already written in the introduction, we will follow the proof of \cite{IS2009} but refine it a bit for large time using the approach of \cite{AMPVZ}. We will conclude this section by obtaining a $L^p-L^q$ estimate for the propagator using the previous pointwise estimate and the Kunze-Stein inequality.\\

 We begin by giving an expression of the propagator with the help of the Fourier transform defined in the previous section. Let $\psi$ be a solution to the linear equation
$$i\partial_t \psi - (\Delta_{\Hy^N}+\rho^2)^2 \psi +\beta (\Delta_{\Hy^N} +\rho^2) \psi =0.$$
By definition, we have, for $f$ smooth enough,
$$\widehat{\Delta_{\Hy^N} f} (\lambda ,\omega)=\int_{\Hy^N} f(x) \Delta_{\Hy^N} h_{\lambda ,\omega} (x)\, d\mu =-(\lambda^2 +\rho^2) \hat{f} (\lambda ,\omega),$$
so applying the Fourier transform to the equation, we obtain
$$i\partial_t \hat{\psi}- \lambda^4 \hat{\psi} - \beta \lambda^2 \hat{\psi} =0 . $$
Its solution is given by
\[
\hat{\psi}(t,\lambda,\omega)=e^{-it\left(\lambda^4+\beta \lambda^2\right)}.
\]
By the Fourier inversion formula, we find
$$\psi(t,x)=\int_{\R}  e^{-it (\lambda^4 +\beta \lambda^2 ) } \Phi_\lambda (r) |c(\lambda)|^{-2}\, d\lambda,  $$
where $r=r(x)$.
For this quantity to be well-defined, we regularize it by
the Fourier multiplier $\lambda\mapsto e^{-\varepsilon^2\lambda^4}$, for $\varepsilon>0$ small enough. So, we will consider 
$$I_\varepsilon (t,x) =\int_{\R}  e^{-it (\lambda^4 +\beta \lambda^2 )-\varepsilon^2 \lambda^4 } \Phi_\lambda (r) |c(\lambda)|^{-2}\, d\lambda. $$
We will obtain pointwise estimate for this propagator. Compared to the one in the Euclidean space, we see that it decays exponentially in space. The behavior for short time is the same whereas it has a slower decay for large time. We will see in the following that to get global existence, we only need to have a decay better than $t^{-1}$ for large time which in our case only seems to be available when $\beta>0$. 

\begin{prop}
\label{propestv}
We have, if $\beta>0$,
$$|I_\varepsilon(t,x)|\leq C \begin{cases}|t|^{-N/4} r^{(N+5)/6} e^{-(N-1)r/2},&\mbox{if }\ |t|\leq 1; \\ |t|^{-3/2}  r^{(N+5)/6} e^{-(N-1)r/2},&\mbox{if }\ |t|\geq 1. \end{cases}$$
If $\beta =0$, then there holds
$$|I_\varepsilon(t,x)|\leq C \begin{cases}|t|^{-N/4} r^{(N+5)/6} e^{-(N-1)r/2},&\mbox{if }\ |t|\leq 1; \\ |t|^{-3/4}  r^{(N+5)/6} e^{-(N-1)r/2},&\mbox{if }\ |t|\geq 1, \end{cases}$$
whereas if $\beta<0$, then we get 
$$|I_\varepsilon(t,x)|\leq C \begin{cases}|t|^{-N/4} r^{(N+5)/6} e^{-(N-1)r/2},&\mbox{if }\ |t|\leq r \ or\ |t|\leq 1; \\ |t|^{-1/2}  r^{(N+5)/6} e^{-(N-1)r/2},&\mbox{if }\ |t|\geq 1\ or\ |t|\geq r. \end{cases}$$
\end{prop}

\begin{proof}
We follow closely the proof of \cite[Lemma 3.2]{IS2009}. 
We start by recalling the estimate 
\begin{equation}
\label{cest}
\big|\partial_{\lambda}^{\alpha}(\lambda^{-1} c(\lambda)^{-1})\big| \leq C (1+|\lambda|)^{\rho-1 - \alpha}
\end{equation}
for $\alpha\in [0,N+2]\cap \Z$ and $\lambda\in\R$ from \cite[Prop. A1]{I-JFA2000}.

We will also need some estimates for the function $\Phi_\lambda$ contained in \cite[Prop. A2]{I-JFA2000}. When $r\geq 1/10$, we rewrite $\Phi_\lambda$ as
\begin{equation}
\label{phismall}
\Phi_\lambda (r)=e^{-\rho r} \big(e^{i\lambda r} c(\lambda) m_1 (\lambda ,r) +e^{-i\lambda r} c(-\lambda ) m_1 (-\lambda ,r)\big),
\end{equation}
where, for $\alpha \in [0,N+2] \cap \Z$ and $\lambda \in \R$, 
\begin{equation}
\label{m1est}
|\partial_\lambda^\alpha m_1 (\lambda ,r)| \leq C (1+|\lambda|)^{-\alpha}.
\end{equation}
On the other hand, if $r\leq 1$, $\Phi_\lambda$ can be rewritten as
\begin{equation}
\label{philarge}
\Phi_\lambda (r)= e^{i\lambda r} m_2 (\lambda ,r)+ e^{-i\lambda r} m_2 (-\lambda ,r),
\end{equation}
where, for $\alpha \in [0,N+2] \cap \Z$ and $\lambda \in \R$,
\begin{equation}
\label{m2est}
|\partial_\lambda^\alpha m_2 (\lambda ,r)| \leq C (1+r|\lambda|)^{-\rho} (1+|\lambda|)^{ - \alpha}.
\end{equation}
We will also use repeatedly the following two standard  estimates for oscillatory integrals: if $A,B,C \in \R$, $m\in C^1 (\R)$ is a compactly supported function, and $\tilde{m} : \R \rightarrow \mathbb{C}$ is a smooth function supported in $[-2,-1/2] \cup [1/2 ,2]$, we have, by the van der Corput lemma,
\begin{equation}
\label{VdC1}
\Big|\int_{\R} e^{i(A\lambda^4 +B\lambda^2 +C\lambda)}m (\lambda)\, d\lambda \Big|\leq \frac{c}{ A_i} \int_{\R} |\partial_\lambda m (\lambda)|\, d\lambda,
\end{equation}
where $A_1= |A|^{1/4}$ if $\sgn (A) \neq \sgn(B)$ or $A_2 = \max \{|A|^{1/4},|B|^{1/2}\}$ if $\sgn(A) = \sgn(B)$, and by integrating by parts $k\in \N$ times
\begin{equation}
\label{VdC2}
\Big|\int_{\R} e^{i(A\lambda^4 +B\lambda^2 +C\lambda)}\tilde{m} (\lambda)\, d\lambda \Big| 
\leq 
\frac{c}{D+1}\sum_{j=0}^k \int_{\R} |\partial_\lambda^j \tilde{m} (\lambda)|\, d\lambda,
\end{equation}
where 
\[
D=\min_{|\lambda| \in [1/2 ,2]}\big|4A\lambda^3 +2B\lambda +C\big|^k
\]
which is assumed to be positive.

We fix a smooth even function $\eta_0 : \R \rightarrow [0,1]$ supported in $[-2,-1/2] \cup [1/2 ,2]$ such that, for any $\lambda \in \R \setminus \{0\}$,
\begin{equation}
\label{eta0}
\sum_{j\in \Z} \eta_0 (\lambda /2^j)=1.
\end{equation}
Let $\eta_j (\lambda )= \eta_0 (\lambda /2^j)$ and $\eta_{\leq j}= \sum_{j^\prime \leq j} \eta_{j^\prime}$, for any $j\in \Z$. In view of \eqref{phismall} and \eqref{philarge}, we will split our estimates depending on the value of $r$.\\
 We begin by considering the case $r\geq 1$. In view of \eqref{cl-2}, \eqref{phismall} and \eqref{eta0}, we need to estimate
 \begin{align*}
&\Big|\int_{\R} e^{i\Psi (\lambda ) -\varepsilon^2 \lambda^4 } \eta_{\leq J}(\lambda)m_1 (\lambda ,r) c(-\lambda)^{-1}\,d\lambda \Big| \\
&\quad +\sum_{j>J}\Big|\int_{\R} e^{i\Psi (\lambda ) -\varepsilon^2 \lambda^4 } \eta_{j}(\lambda)m_1 (\lambda ,r) c(-\lambda)^{-1}\,d\lambda\Big| \\
& = I+\sum_{j>J}II_j,  
\end{align*}
where $\Psi(\lambda)= t (-\lambda^4 -\beta\lambda^2) +r\lambda$ and $J$ is the smallest integer such that $2^J \geq 2^{10}\lambda_{r,t}$, with 
$$\lambda_{r,t}=\begin{cases} (r/|t|)^{1/3},&\mbox{if }\  r/|t| \geq 1 ;\\ 
r/|t|,&\mbox{if }\  r/|t| \leq 1. \end{cases}$$
Notice that $\lambda_{r,t}$ is chosen in view of the asymptotics of the roots of $\Psi^\prime$ when $\beta\geq 0$. We will deal with the case $\beta <0$ at the end of the proof (in this case, we have more critical points). Let us point out that thanks to the above definition of $J$, the phase only has critical points in the set where $\eta_{\leq J}$ is supported i.e. in the term $I$. In view of \eqref{VdC1} and \eqref{VdC2}, we will need to estimate the partial derivatives  of $m_1 (\lambda ,r) c(-\lambda)^{-1}$ with respect to $\lambda$. Using \eqref{cest} and \eqref{m1est}, we find, for any $k\in \N \cap [0,N+2]$,
\begin{align}
\label{estder}
\big|\partial_{\lambda}^{k}\big(m_1 (\lambda ,r) c(-\lambda)^{-1}\big)\big| & = \big|\partial_{\lambda}^{k}\big(m_1 (\lambda ,r)\lambda\lambda^{-1} c(-\lambda)^{-1}\big)\big|\nonumber\\
&\leq C \sum_{k_1 +k_2 +k_3 =k} 
\big|\big(\partial_{\lambda}^{k_1}m_1(\lambda,r)\big)\,\big(\partial_{\lambda}^{k_2}\lambda\big)\, \partial_{\lambda}^{k_3}\big(-\lambda^{-1} c(-\lambda)^{-1}\big)\big|\nonumber\\
&\leq C \sum_{k_1+k_2+k_3=k} (1+|\lambda|)^{\rho -1 - k_1 - k_3} \big|\partial_{\lambda}^{k_2}\lambda\big|\nonumber\\
&\leq C (|\lambda| (1+|\lambda|)^{\rho -1 - k}+\one_{k\geq 1} (1+|\lambda|)^{\rho -k} ).
\end{align}
Using \eqref{estder} with $k=0,1$, the fact that the support of $\eta_{\leq J} \subset [-2^J,2^J]$ and \eqref{VdC1}, we get in case $\beta>0$
\begin{align*}
I&=\Big|\int_\R e^{i\Psi(\lambda)}e^{-\varepsilon^2 \lambda^4} \eta_{\leq J}(\lambda) m_1 (\lambda ,r) c(-\lambda )^{-1}\, d\lambda\Big|\\
&\leq C \min \big(|t|^{-1/2},|t|^{-1/4}\big) \int_{\R} 
\big|\partial_\lambda \big(\eta_{\leq J}(\lambda) m_1 (\lambda ,r) c(-\lambda)^{-1}\big)\big|\, d\lambda  \\
&\leq C  \min \big(|t|^{-1/2},|t|^{-1/4}\big) \int_{\R} \big|\big(-\lambda\partial_\lambda \eta_{\leq J}(\lambda)\big) m_1 (\lambda ,r)\big(-\lambda^{-1} c(-\lambda)^{-1}\big)
\\
&\qquad + \eta_{\leq J}(\lambda) \partial_\lambda \big(m_1 (\lambda ,r) c(-\lambda)^{-1}\big)
\big|\, d\lambda\\
&\leq C \min \big(|t|^{-1/2},|t|^{-1/4}\big)\Big(  2^J (1+2^J)^{\rho -1} + \int_0^{2^J} \big(\lambda (1+ \lambda)^{\rho -2}+ (1+\lambda)^{\rho -1} \big)\,d\lambda\Big)\\
&\leq C  \min \big(|t|^{-1/2},|t|^{-1/4}\big) 
\begin{cases} 
2^{\rho J},&\mbox{if }\ J\geq 0;\\ 
2^J ,&\mbox{if }\ J\leq 0 . 
\end{cases}
\end{align*}
 Above we have ignored the term $e^{-\varepsilon^2\lambda^4}$ since it has no effect on the result. In all the following, we are going to do the same.
So, if $r/|t| \geq 1$ which implies that $J\geq 0$, we deduce that
\begin{align}
\label{1412e1}
I &\leq C \min \big(|t|^{-1/2},|t|^{-1/4}\big) 2^{\rho J}\nonumber \\
&\leq C\min \big(|t|^{-1/2},|t|^{-1/4}\big)  (r/|t|)^{\rho /3} \nonumber\\
&\leq Cr^{\rho /3}
\begin{cases} 
|t|^{-(2N+1)/12},&\mbox{if }\ |t|\leq 1;\\  
|t|^{-(N+2)/6},&\mbox{if }\ |t|\geq 1 ,
\end{cases}
\end{align} 
whereas, if $r/|t|\leq 1$ which implies that $|t|\geq 1$, we get
\begin{equation}
\label{1412e2}
I\leq C\min \big(|t|^{-1/2},|t|^{-1/4}\big) 2^{ J} 
\leq C r|t|^{-3/2}.
\end{equation}
We also used that $2^J \leq 2^{11}\lambda_{r,t}$.

If $\beta\le 0$, we have instead
\begin{align*}
I &=\Big|\int_\R e^{i\Psi(\lambda)}e^{-\varepsilon^2 \lambda^4} \eta_{\leq J}(\lambda) m_1 (\lambda ,r) c(\lambda )^{-1}\, d\lambda\Big|\\
&\leq C  |t|^{-1/4} 
\begin{cases} 2^{\rho J},&\mbox{if }\  J\geq 0,\\ 
2^J ,&\mbox{if }\  J\leq 0 . \end{cases}
 \end{align*}
This implies that
\begin{equation}
\label{1412e3}
I\le Cr^{\rho/3}|t|^{-(2N+1)/12}
\end{equation}
if $r/|t|\geq 1$, and
\begin{equation}
\label{1412e4}
I\leq C  r|t|^{-5/4}, 
\end{equation}
if $r/|t| \leq 1$.

Next, we are going to estimate the $II_j$ terms. We recall that by definition of $J$, the phase does not have any critical point on the support of these integrals. On the other hand, 
for $j\geq J+1$ and $\beta> 0$, we deduce from \eqref{VdC2}, \eqref{estder} and a change of variables, that 
\begin{align*}
II_j=&\Big|\int_\R e^{i\Psi(\lambda)}e^{-\varepsilon^2\lambda^4} \eta_{j}(\lambda) m_1 (\lambda ,r) c(-\lambda )^{-1}\, d\lambda\Big|\\
&\leq \frac{C}{1+2^{Nj}\min_{|\lambda|\in [2^{j-1},2^{j+1}] } |\partial_\lambda\Psi(\lambda)|^N} 2^j \int_{1/2}^2 \sum_{k=0}^{N} 
\big|\partial_\lambda^k \big(\eta_0(\lambda)m_1(2^j \lambda,r)c(-2^j\lambda )^{-1}\big)\big|\, d\lambda \\
&\leq \frac{C}{1+2^{Nj}(r+(2^{3j}+2^{j}) |t|)^N}  \sum_{k=0}^{N} 2^j  \big(2^j (1+2^j)^{\rho -1 -k} +(1+2^j)^{\rho - k}\big)\\ 
&\leq   \frac{C}{1+2^{Nj}(r+(2^{3j}+2^{j}) |t|)^N} \begin{cases}2^{j (\rho +1)},&\mbox{if }\ j\geq 0, \\
2^j ,&\mbox{if }\  j\leq 0 . \end{cases}
\end{align*}
If $J < 0$, we have
$$\sum_{j=J+1}^{0}\dfrac{2^j}{(1+(2^{4j} +2^{2j} )|t| )^N} 
\leq \int_0^1 \dfrac{2}{(1+x^2 |t|)^N}\, dx 
\leq C(1+|t|)^{-1/2}.$$
Notice that in the case $j\geq 0$, we have $\sum_{j\geq 0} II_j \leq 1/(1+|t|)^N$.
In case $\beta=0$, we have 
\[
 II_j\leq   \frac{C}{1+2^{jN}(r+2^{3j} |t|)^N} \begin{cases}2^{j (\rho +1)},&\mbox{if }\ j\geq 0, \\
2^j ,&\mbox{if }\  j\leq 0  \end{cases}
\]
and 
\[
\sum_{j=J+1}^{0}\dfrac{2^j}{(1+2^{4j}|t| )^N} 
\leq C(1+|t|)^{-1/4}
\] for $J<0$.
Combining the previous bounds, \eqref{1412e1} and \eqref{1412e2}, we obtain that, for $\beta > 0$,
$$I+\sum_{j>J}II_j  \leq C (|t|^{-N/4} +\max\{|t|^{-1/2},|t|^{-(N+2)/6}\}) (1+r)^{\rho/3 +1},$$
whereas, if $\beta=0$, using \eqref{1412e3} and \eqref{1412e4}, we get
$$I+\sum_{j>J}II_j  \leq C (|t|^{-N/4} +|t|^{-1/4}) (1+r)^{\rho/3 +1},$$
Next, we consider the case $r\leq 1$. We will need the following estimate, where we use \eqref{m2est},
\begin{align}
\label{estm22}
&\big|\partial_\lambda^k\big(m_2 (\lambda ,r) | c(\lambda)|^{-2}\big)\big|
=\big|\partial_\lambda^k\big(m_2 (\lambda ,r) c(\lambda)^{-1}c(-\lambda)^{-1}\big)\big|
\nonumber\\
&\quad \leq C \sum_{k_1 +k_2 +k_3 +k_4=k}
\big|\partial_\lambda^{k_1} m_2(\lambda,r)\partial_\lambda^{k_2}(\lambda^2)\partial_\lambda^{k_3}\big(\lambda^{-1}c(\lambda)^{-1}\big)\partial_\lambda^{k_4}\big(-\lambda^{-1} c(-\lambda)^{-1}\big)\big|\nonumber\\
&\quad\leq C \sum_{k_1+k_2+k_3+k_4=k} (1+r |\lambda|)^{-\rho}  (1+|\lambda|)^{2(\rho -1) - k_1 - k_3-k_4} \big|\partial_\lambda^{k_2}(\lambda^2)|\nonumber\\
&\quad\leq C (1+r|\lambda|)^{-\rho} (1+|\lambda|)^{2(\rho -1)}   (  |\lambda|^2 (1+|\lambda|)^{- k}+ \one_{k\geq 1} |\lambda| (1+|\lambda|)^{1- k} \\
&\qquad +\one_{k\geq 2}  (1+|\lambda|)^{2- k}).\nonumber
\end{align}


Using the previous estimate and \eqref{VdC1}, we get in case $\beta > 0$,
\begin{align*}
\tilde{I}=&\left|\int_\R e^{i\Psi(\lambda)}e^{-\varepsilon^2 \lambda^4} \eta_{\leq J}(\lambda) m_2 (\lambda ,r)| c(\lambda )|^{-2}\, d\lambda\right|\\
&\leq C \min \big(|t|^{-1/2},|t|^{-1/4}\big) \int_{\R}\left|\partial_\lambda \big(\eta_{\leq J} m_2 (\lambda ,r) |c(\lambda)|^{-2} \big)\right|\, d\lambda     \\
&\leq C  \min \big(|t|^{-1/2},|t|^{-1/4}\big) \int_{\R}\Big| \big(\partial_\lambda \eta_{\leq J}\big) m_2 (\lambda ,r) |c(\lambda)|^{-2} \\
&\qquad + \eta_{\leq J} \partial_\lambda \big(m_2 (\lambda ,r) |c(-\lambda)|^{-2}\big)\Big|\, d\lambda\\
&\leq C \min \big(|t|^{-1/2},|t|^{-1/4}\big)\Big[  2^{2J}(1+r 2^J)^{-\rho} (1+2^J)^{2(\rho -1)} \\
&\quad + \int_0^{2^J} (1+r\lambda)^{-\rho} (1+ \lambda)^{2(\rho -1)} \big(\lambda^{2} (1+\lambda)^{-1} +\lambda \big)\,d\lambda \Big]\\
&\leq C \min \big(|t|^{-1/2},|t|^{-1/4}\big) \begin{cases} 2^{\rho J}r^{-\rho},&\mbox{if }\ J\geq 0,\ r2^J \geq 1\\ 2^{2\rho J}   ,&\mbox{if }\ J\geq 0,\ r2^J \leq 1
\\ 2^{2J}  ,&\mbox{if }\ J\leq 10 . \end{cases}
 \end{align*}

So, when $r\leq 1$, $r/|t| \geq 1$ and $2^J r\geq 1$, which implies $r^4 \geq C|t|$, we have 
\begin{align*}
 \tilde{I}&\leq C \min \big(|t|^{-1/2},|t|^{-1/4}\big) 2^{\rho  J} r^{-\rho} \\
&\leq C |t|^{-1/4}  \big(r/|t|\big)^{\rho/3} r^{-\rho}\\
& \leq C |t|^{-1/4 - \rho /2}=C |t|^{-N/4}.  
\end{align*}

When $r\leq 1$, $r/|t| \geq 1$ and $2^J r\leq 1$, which imply $r^4 \leq C |t|$, we get
\begin{align*}
\tilde{I}& \leq C \min \big(|t|^{-1/2},|t|^{-1/4}\big) 2^{2\rho  J}\\
& \leq  C |t|^{-1/4}\big(r/|t|\big)^{2\rho/3} \leq C |t|^{-1/4}  |t|^{ \rho /6} |t|^{-2\rho /3}=C |t|^{-N/4}.   
\end{align*}

Finally, when $r\leq 1$ and $r/|t| \leq 1$, we find
\begin{align*}
\tilde{I} &\leq C \min \big(|t|^{-1/2},|t|^{-1/4}\big) 2^{2 J} \leq C \min \big(|t|^{-1/2},|t|^{-1/4}\big) \big(r/|t|\big)^2\\
& \leq C (|t|^{-3/2}+|t|^{-N/4}),
\end{align*}
since $|t|^{-1/4-2 }\leq |t|^{-3/2}$ if $|t|\geq 1$ and $|t|^{-1/2 }\leq |t|^{-N/4}$, if $|t|\leq 1$.
The same holds if $\beta = 0$; instead of having $\min (|t|^{-1/2},|t|^{-1/4})$, we have $|t|^{-1/4}$, but the estimates still hold. So overall, we showed that
$$\tilde{I} \leq C (|t|^{-3/2}+|t|^{-N/4}). $$
To conclude the proof, we consider the case where 
$j\ge J+1$ and $\beta> 0$. 
We have, using \eqref{VdC2} and \eqref{estm22},
\begin{align}\label{esttilde2}
\widetilde{II}_j=
&\Big|\int_\R e^{i\Psi(\lambda)}e^{-\varepsilon^2 \lambda^4} \eta_{j}(\lambda) m_2 (\lambda ,r) |c(\lambda )|^{-2}\, d\lambda\Big|\nonumber\\
&\leq  \frac{C}{1+2^{Nj}(r+(2^{3j}+2^j) |t|)^N} 2^j \int_{1/2}^2 \sum_{k=0}^{N}
\big|\partial_\lambda^k\big(m_2 (2^j\lambda ,r) | c(2^j\lambda)|^{-2}\big)\big|\, d\lambda \nonumber\\
&\leq  \frac{C}{1+2^{Nj}( r+(2^{3j}+2^j)|t|)^N} \sum_{k=0}^{N} 2^j  (1+ r 2^j)^{-\rho} (1+2^j)^{2 (\rho -1)} \nonumber\\
&\qquad \cdot    \big(2^{2j} (1+2^j)^{ -k} + 2^j(1+2^j)^{1- k} +(1+2^j)^{2-k}\big)\nonumber\\
&\leq  \frac{C}{1+2^{Nj}( r+(2^{3j}+2^j)|t|)^N} \begin{cases} 2^{(\rho +1) j}r^{-\rho},&\mbox{if }\ j\geq 0,\ r2^j \geq 1,\\ 2^{j (2\rho +1) }   ,&\mbox{if }\ j\geq 0,\ r2^j \leq 1,
\\ 2^{3j} ,&\mbox{if }\ j\leq 0 . \end{cases}
\end{align}
On the other hand, we also have 
\begin{align*}
&\Big|\int_\R e^{i\Psi(\lambda)}e^{-\varepsilon^2 \lambda^4} \eta_{j}(\lambda) m_2 (\lambda ,r) |c(\lambda) |^{-2}\, d\lambda\Big|\\
&\leq C 2^j \int_{1/2}^2 |m_2 (2^j \lambda ,r)| |c(2^j \lambda) |^{-2}\, d\lambda \\
&\leq C 2^{3j}, 
\end{align*}
if $j\leq 0$.
Thus, letting $J_1$ be the largest integer such that 
$2^{J_1}\le 1/|t|^{1/2}$, the previous estimate yields to 
\[
\sum_{j=J+1}^{J_1} \widetilde{II}_j\leq C/|t|^{3/2},
\]
 if $J\leq 0$. On the other hand, we have, using \eqref{esttilde2}, 
$$\sum_{j=J_1}^0 \frac{C2^{3j}}{( 1+(2^{4j}+2^{2j}) |t|)^N}\leq \frac{C}{|t|^{3/2}} \sum_{j=J_1}^0 2^{(j-J_1)(3-2N)}
\leq C /|t|^{3/2}. $$
Overall, we showed that
\begin{equation}\label{j+1to0}
\sum_{J+1}^0 \widetilde{II}_j\le C/|t|^{3/2}
\end{equation}
if $J\le 0$.  
In case $\beta=0$ we have $2^{4j}$ instead of $(2^{2j} +2^{4j})$ in the estimates above. Hence 
\eqref{j+1to0} holds for $\beta=0$, too, and consequently
\[
\tilde{I}+\sum_{j>J}\widetilde{II}_j \le
C \big(|t|^{-3/2}+|t|^{-N/4}\big)
\] 
if $\beta\ge 0$.
When $|t|$ is very large, we can improve the previous estimates by using the method of \cite{AMPVZ}. Indeed, instead of using a dyadic decomposition, we write
\begin{align*}
I_\varepsilon (t,x)&= \int_{\R} \chi^0_t (\lambda)  e^{-(it+\varepsilon^2) (\lambda^4 +\tilde{\beta} \lambda^2 ) } \Phi_\lambda (r) |c(\lambda)|^{-2}\, d\lambda  \\
&+ \int_{\R} \chi^\infty_t (\lambda)  e^{-(it+\varepsilon^2) (\lambda^4 +\tilde{\beta} \lambda^2 ) } \Phi_\lambda (r) |c(\lambda)|^{-2}\, d\lambda\\
&= I_0 (t,r) +I_\infty (t,r),
\end{align*}
where $\chi_t^0 (\lambda) = \chi (\sqrt{t} |\lambda|)$ is an even cut-off function such that $\chi (\lambda)=1$, if $|\lambda| \leq 1$ and $\chi (\lambda)=0$, if $|\lambda| \geq 2$ and $\chi_t^\infty =1-\chi_t^0$. Using that   $|\Phi_\lambda (r)|\leq |\Phi_0 (r)|\leq C e^{-\rho r}$ and $|c(\lambda)|^{-2}\leq C\lambda^2$, we find that
$$|I_0 (t,r)|\leq |t|^{-3/2}e^{-\rho r}. $$
On the other hand, after integrating by parts $\tilde{M} (=k_0+k_1+k_2+k_3)$ times, we obtain
\begin{align*}
e^{\rho r/2}&|I_\infty (t,r)| \leq C |t|^{-\tilde{M}} \int_{\R} \Big|e^{-it (\lambda^4 +\beta\lambda^2 )}  \left[ \dfrac{\partial}{\partial \lambda} \circ \dfrac{1}{\lambda}\right]^{\tilde{M}} \big(\chi_t^\infty (\lambda ) |c(\lambda)|^{-2} e^{ir}\big)\Big|\, d\lambda\\
&\leq C |t|^{-\tilde{M}} \int_{\R} \sum_{k_0 +k_1+ k_2 +k_3=\tilde{M}} |t|^{k_0 /2} |\lambda|^{-\tilde{M}-k_1}  r^{k_3} \begin{cases} |\lambda|^{2},&\mbox{if }\ |\lambda|\leq 1,\\ |\lambda|^{N-1},&\mbox{if }\ |\lambda|\geq 1  \end{cases} d\lambda \\
&\leq  \mathop{\sum_{k_0 +k_1+ k_2 +k_3=\tilde{M}}}_{k_0\ge 1}
|t|^{-3 /2}  r^{k_3}\\
&+ \sum_{k_1+ k_2 +k_3=\tilde{M}}  |t|^{-\tilde{M}} r^{k_3} \Big[\int_{|t|^{-1/2}\leq |\lambda|\leq 1} |\lambda|^{2-\tilde{M}-k_1}\, d\lambda +\int_{|\lambda|\geq 1} |\lambda|^{N-1 -\tilde{M}-k_1}\,d\lambda \Big]\\
&\leq |t|^{-3/2} (1+r)^{\tilde{M}},
\end{align*}
provided that $\tilde{M}>\max \{N, 3/2\}$. Here we used that 
$|\lambda|\approx |t|^{-1/2}$ if $k_0 \geq 1$, and $4|\lambda|^3 + 2\beta|\lambda| \geq C |\lambda| $ if $\beta>0$. If $\beta=0$, we take $\chi_t^0 (\lambda) = \chi (t^{1/4} |\lambda|)$, then we have
\begin{align*}
e^{\rho r/2}&|I_\infty (t,r)| \leq C |t|^{-\tilde{M}} \int_{\R} \Big|e^{-it (\lambda^4 +\beta\lambda^2 )}  \left[ \dfrac{\partial}{\partial \lambda} \circ \dfrac{1}{\lambda^3}\right]^{\tilde{M}} \big(\chi_t^\infty (\lambda ) |c(\lambda)|^{-2} e^{ir}\big)\Big|\, d\lambda\\
&\leq C |t|^{-\tilde{M}} \int_{\R} \sum_{k_0 +k_1+ k_2 +k_3=\tilde{M}} |t|^{k_0 /2} |\lambda|^{-3 \tilde{M}-k_1  }  r^{k_3} \begin{cases} |\lambda|^{2},&\mbox{if }\ |\lambda|\leq 1,\\  |\lambda|^{N-1},&\mbox{if }\ |\lambda|\geq 1  \end{cases} d\lambda \\
&\leq \mathop{\sum_{k_0 +k_1+ k_2 +k_3=\tilde{M}}}_{k_0\ge 1}
|t|^{-3 /4}  r^{k_3}\\
&+ \sum_{k_1+ k_2 +k_3=\tilde{M}}  |t|^{-\tilde{M}} r^{k_3} \Big[\int_{|t|^{-1/4}\leq |\lambda|\leq 1} |\lambda|^{2-3\tilde{M}-k_1}\, d\lambda +\int_{|\lambda|\geq 1} |\lambda|^{N-1 -3\tilde{M}-k_1}\,d\lambda \Big]\\
&\leq |t|^{-3/4} (1+r)^{\tilde{M}}.
\end{align*}

If $\beta<0$, we only have to consider two extra cases which correspond to the situation where $\partial_\lambda\Psi(\lambda)$ degenerates, i.e. when $r/|t|$ is small and $|\lambda|$ is close to 
$\sqrt{-\beta/2}=\lambda_0$  and when $|r/|t| -\lambda_1|$ is small, where  $\lambda_1 =\sqrt{-\beta/6}$, and $\lambda$ is close to $\lambda_1$. Let $\chi$ be a cut-off function supported in $[-1,1]$
and $\tilde{M}$ a sufficiently large constant. When $r/|t|$ is small, we find, by using the van der Corput lemma see \cite[Step 6]{MR1745182}, 
$$\Big|\int_\R e^{i\Psi (\lambda )} \chi \big(\tilde{M} (\lambda -\lambda_0 ) \big) \big(m_2 (\lambda ,r)|c(\lambda)|^{-2}+m_1(\lambda,r) c(-\lambda)^{-1} \big)\, d\lambda \Big|\leq Ct^{-1/2} (1 +r)^{-\rho}.$$
When $|r/|t| -\lambda_1|$ is small, we find that
$$|\int_\R e^{i\Psi (\lambda )} \chi \big(\tilde{M} (\lambda -\lambda_1 ) \big)  \big(m_2 (\lambda ,r) |c(\lambda)|^{-2}+m_1(\lambda,r) c(-\lambda)^{-1} \big)\, d\lambda  |\leq Ct^{-1/3} (1 +r)^{-\rho} .$$
Notice that, $r\approx t$ in this case, so using the exponential decay in $r$, we find the result.


\end{proof}

Thanks to the previous proposition, 
we are able to estimate the norm of $I_\varepsilon$ in the Lorentz space $L^{q,\theta}(\mathbb{H}^N)$, with $q>2$ and $\theta\geq 1$. Let us recall that the norm of this space is defined by
$$\|f\|_{L^{q,\theta}(\mathbb{H}^N)}= \begin{cases}\big( \int_0^\infty  (s^{1/q} f^\ast (s))^\theta \frac{ds}{s}\big)^{1/\theta},&\mbox{if }\ 1\leq \theta <\infty ,\\ \sup_{s>0} s^{1/q} f^\ast (s),&\mbox{if }\ \theta=\infty , \end{cases}$$
where $f^\ast$ is the decreasing rearrangement of $f$. If $f$ is a positive radial decreasing function in $\mathbb{H}^N$ then $f^\ast = f \circ V^{-1}$ where
$$V(r)=C \int_0^r \sinh^{N-1}s\, ds$$
is the volume of a geodesic ball of radius $r$.
So, using this last fact, we can rewrite the norm as
$$\|f\|_{L^{q,\theta}(\mathbb{H}^N)} =\begin{cases} \big(\int_0^\infty \big(V(r)^{1/q} f(r)\big)^{\theta} \frac{V^\prime (r)}{V(r)}\, dr \big)^{1/\theta}, &\mbox{if }\ 1\leq \theta <\infty ,\\ \sup_{s>0} V(s)^{1/q} f (s),&\mbox{if }\ \theta=\infty . \end{cases}$$
Using the asymptotics of $V$, we see that this norm is equivalent to
$$\|f\|_{L^{q,\theta}(\mathbb{H}^N)} \approx \begin{cases} \big(\int_0^1  f(r)^{\theta} r^{\tfrac{\theta N }{q}-1}\, dr \big)^{1/\theta}+ \big(\int_1^\infty  f(r)^{\theta} e^{t\frac{\theta (N-1) }{q}} \,dr \big)^{1/\theta} , &\mbox{if }\ 1\leq \theta <\infty ,\\ \sup_{0<s<1} s^{N/q} f (s)+\sup_{s\geq 1} e^{\tfrac{N-1}{q}s} f(s),&\mbox{if }\ \theta=\infty . \end{cases}$$

We are now in position to estimate the $L^{q,\theta}(\mathbb{H}^N)$ norm of $v_t=v(t,\cdot)=\lim_{\varepsilon \rightarrow 0}I_\varepsilon (t,.)$. 

\begin{prop} 
Let $q>2$ and $1\leq \theta\leq \infty$. Then 
we have 
\[
\|v_t\|_{L^{q,\theta}(\mathbb{H}^N )}\leq C\begin{cases} |t|^{-N/4},&\mbox{if }\ |t|\leq 1,\\   |t|^{-3/2},&\mbox{if }\ |t|>1 \ and\ \beta>0,\\  |t|^{-3/4},&\mbox{if }\ |t|>1 \ and\ \beta=0.  \end{cases}
\]
\end{prop}

\begin{proof}
Using Proposition \ref{propestv}, we see that
$$|v(t,r)|\le C r^{\frac{N+5}{6}} e^{-\frac{N-1}{2}r} \begin{cases}|t|^{-N/4},&\mbox{if }\ |t|\leq 1,\\ |t|^{-3/2},&\mbox{if }\ |t|\geq 1\ and\ \beta>0,\\ |t|^{-3/4},&\mbox{if }\ |t|>1 \ and\ \beta=0.  \end{cases}$$
Taking 
\[
f(r)=r^{\frac{N+5}{6}} e^{-\frac{N-1}{2}r},
\]
 it is easy to see that $\|f\|_{L^{q,\theta}(\mathbb{H}^N)}\leq C$ since $q>2$. This yields the result.
\end{proof}

Recalling that $P=(\Delta_{\Hy^N} +\rho^2)^2  -\beta (\Delta_{\Hy^N} +\rho^2)$, we obtain an $L^p -L^q$ estimate for $e^{itP}$ by using the Kunze-Stein inequality.
\begin{prop}\label{disp}
 For all $2<q,\tilde{q}\leq \infty$, there exists a constant $C>0$ such that
$$\|e^{it P} \|_{L^{\tilde{q}^\prime } (\mathbb{H}^N) \rightarrow L^q  (\mathbb{H}^N)} \leq C\begin{cases} |t|^{-\max ( 1 - \frac{2}{q}, 1 - \frac{2}{\tilde{q}})\frac{N}{4}},&\mbox{ if}\ |t|\leq 1 ,\\ |t|^{-3/2},&\mbox{ if}\ |t|\geq 1\ and\ \beta>0,\\ |t|^{-3/4},&\mbox{ if}\ |t|\geq 1\ and\ \beta=0. \end{cases} $$
\end{prop}
\begin{proof}
By Young's inequality for convolutions and recalling that $L^{p,p}(\mathbb{H}^N)=L^p (\mathbb{H}^N)$, we have, for small $|t|$,
$$\|e^{itP} \|_{L^1  (\mathbb{H}^N) \rightarrow L^q  (\mathbb{H}^N)} \leq \|v_t\|_{L^q  (\mathbb{H}^N)}\leq C|t|^{-N/4},\ q>2  .$$
By duality, this implies that, for small $|t|$,
$$\|e^{itP} \|_{L^{q^\prime} (\mathbb{H}^N) \rightarrow L^\infty  (\mathbb{H}^N)} \leq \|v_t\|_{L^q  (\mathbb{H}^N)}\leq C|t|^{-N/4},\ q>2 .$$
On the other hand, we have
$$ \|e^{itP} \|_{L^2  (\mathbb{H}^N)\rightarrow L^2  (\mathbb{H}^N)} \leq  1.$$
Interpolating between the three previous inequalities, we get the result when $|t|\leq 1$. For large $|t|$, we will need the following sharp Kunze-Stein inequality due to Cowling, Meda and Setti \cite{cowling}, \cite{CoMeSe}
$$\|f\ast g\|_{L^q  (\mathbb{H}^N)}\leq \|f\|_{L^{q,1}  (\mathbb{H}^N)}  \|g\|_{L^{q^\prime}  (\mathbb{H}^N)}$$
 for $q>2$.
Using this inequality, we get that, when $|t|$ is large and $q>2$,
$$ \|e^{itP} \|_{L^{q^\prime}  (\mathbb{H}^N) \rightarrow L^q  (\mathbb{H}^N)} \leq  \|v_t\|_{L^{q,1}}\leq C\begin{cases}|t|^{-3/2},&\mbox{ if}\ \beta>0,\\ |t|^{-3/4},&\mbox{ if}\ \beta=0. \end{cases}$$
We also have
$$\|e^{itP} \|_{L^1   (\mathbb{H}^N)\rightarrow L^q  (\mathbb{H}^N)} \leq \|v_t\|_{L^q}\leq C\begin{cases}|t|^{-3/2},&\mbox{ if}\ \beta>0,\\ |t|^{-3/4},&\mbox{ if}\ \beta=0,\end{cases}$$
and, by duality,
$$\|e^{itP} \|_{L^{q^\prime}  (\mathbb{H}^N) \rightarrow L^\infty  (\mathbb{H}^N)} \leq \|v_t\|_{L^q}\leq C\begin{cases}|t|^{-3/2},&\mbox{ if}\ \beta>0,\\ |t|^{-3/4},&\mbox{ if}\ \beta=0. \end{cases}$$
Interpolating between the three previous inequalities, the proposition follows.
\end{proof}

\section{Strichartz estimates and applications on global existence and scattering}\label{secstri}
In this section, we give the proof of Theorems \ref{introstri} and \ref{thmexintro}. We end it by giving a trapping result in the spirit of \cite{MR3338838}.

We say that $(p,q)$ is an admissible pair if 
$$\left(\frac{1}{p},\frac{1}{q}\right)\in  \left\lbrace\left(0,\frac{1}{2} \right)\times \left(0,\frac{1}{2}\right) \colon \frac{4}{p} +\frac{N}{q} \geq \frac{N}{2} \right\rbrace \cup \left\lbrace\left(0,\frac{1}{2}\right) \right\rbrace.$$
Observe that the set of admissible couples in the hyperbolic space is a lot larger than the corresponding one in the Euclidean space (as for the classical Schr\" odinger equation). Our proof of the Strichartz estimate is quite standard and follows the $TT^\ast$ method of Keel-Tao \cite{keel-tao}.
We do not consider the endpoint cases of the Strichartz estimates but they can be obtained as in \cite{keel-tao}.


\begin{thm}
\label{thmstri} 
Assume that $(p,q)$ and $(\tilde{p},\tilde{q})$ are admissible pairs. Let $\psi$ be a solution to
$$\begin{cases} i\partial_t \psi +P \psi=h\mbox{ in}\ \R\times \Hy^N ,\\ \psi(0,\cdot)=\psi_0 \in L^2 (\Hy^N),\end{cases}$$
on $I=[0,T]$. Assume moreover that $|I|\leq 1$ if $\beta\leq 0$. Then there exists a constant $C>0$ such that
\begin{equation}
\label{stri}
\|\psi\|_{L_t^p (I, L_x^q (\Hy^N))}\leq C (\|\psi_0\|_{L^2(\Hy^N)} +\|h\|_{L_t^{\tilde{p}^\prime } (I,L_x^{\tilde{q}^\prime} (\Hy^N))}).
\end{equation}

\end{thm}

\begin{proof}
First, we recall that, using Duhamel's formula, $\psi$ is given by 
$$\psi(t,x)= e^{itP}\psi_0(x) -i\int_0^t e^{i(t-s)P} h(s,x)\,ds.$$
We set
$$Tf= e^{it P}f(x).$$
Its formal adjoint is then given by
$$T^\ast F(x)= \int_{-\infty}^\infty e^{-isP}F(s,x)\, ds.$$
By the $TT^\ast$ method, we know that 
$$\|Tf \|_{L_t^p L_x^q}\leq C \|f\|_{L^2}$$
is equivalent to the boundedness in $L_t^{p^\prime}L_x^{q^\prime}\rightarrow L_{t}^p L_x^q$ of the operator
$$TT^\ast F (t,x)=\int_{-\infty}^\infty e^{i(t-s)P}F(s,x)\, ds.$$
By symmetry, it is sufficient to consider the retarded version of the previous operator
$$ \widetilde{T T^\ast} F(x,t)= \int_0^t e^{i(t-s)P}F(s,x)\, ds. $$
Using the dispersive estimates in Proposition~\ref{disp}, we have that
\begin{align*}
\|TT^\ast F\|_{L_t^p L_x^q}&\leq C \Big\|\int_{|t-s|\geq 1} |t-s|^{-3/2} \|F(s)\|_{L_x^{q^\prime}} \Big\|_{L_t^p}\\
&+C\Big\|\int_{|t-s|\leq 1} |t-s|^{-(1 -\frac{ 2}{q})\frac{N}{4}} \|F(s)\|_{L_x^{q^\prime}} \Big\|_{L_t^p}.
\end{align*}
Let $N_1(s,t)=\one_{ \{|s-t|\leq 1\}}|t-s|^{-1/2 (1/2 - 1/q)N} $. Using Young's inequality, we see that $N_1$ is bounded from $L_s^{p_1}$ to $L_t^{p_2}$ provided that $0\leq 1/p_1 - 1/p_2 \leq 1- (1-\frac{2}{q} )\frac{N}{4}$ and $1<p_1,p_2 <\infty$. In particular, it is bounded from $L_s^{p^\prime}$ to $L_t^{p}$ if $2\leq p<\infty$ and $\frac{1}{p}\geq (1-\frac{2}{q} )\frac{N}{8}$. On the other hand, $N_2 (s,t)=\one_{ \{|s-t|\geq 1\}}|t-s|^{-3/2 } $ is bounded in $L_s^{p^\prime}$ into $L_t^{p}$ provided that $p> 2$. This proves the boundedness of $TT^\ast F$ in $L_t^{p^\prime}L_x^{q^\prime}\rightarrow L_{t}^p L_x^q$. The boundedness in $L_t^{p^\prime}L_x^{q^\prime}\rightarrow L_{t}^{\tilde{p}} L_x^{\tilde{q}}$ can then be obtained as in Section $7$ of \cite{keel-tao}. 
 The case $(p,q)=(\infty ,2)$ is settled by the conservation of the mass.

\end{proof}

Let us point out that to prove the $L^{p^\prime}-L^p$ estimate, for $p>2$, of $N_2$, we use the fact that $v_t$ decays faster than $|t|^{-1}$ when $|t|$ is large.

We consider 
\begin{equation}
\label{4NLSeq}
\begin{cases}i\partial_t \psi +P \psi =F(\psi)\mbox{ in}\ \R \times \Hy^N ,\\ \psi(0,\cdot)=\psi_0\in X(\Hy^N),\end{cases}
\end{equation}
where $X (\Hy^N)= L^2 (\Hy^N)$ or $H^2 (\Hy^N)$, and $F$ is a function satisfying 
\[
|F(\psi)|\leq C |\psi|^\gamma\quad\text{and}\quad 
|F(\psi)-F(\tilde\psi)|\leq C |\psi-\tilde \psi | \big(|\psi|^{\gamma -1}+|\tilde \psi|^{\gamma -1}\big),
\]
 for some $\gamma>1$. Thanks to Theorem \ref{thmstri}, we will be able to obtain some global well-posedness results for small initial data in $L^2 (\Hy^N)$ or $H^2 (\Hy^N)$ by using a standard fixed point approach.

\begin{thm}
\label{thmexl2}
 If $1<\gamma \leq 1+ \tfrac{8}{N}$ and $\beta>0$, then there exists a constant $c>0$ such that if $\|\psi_0\|_{L^2 (\mathbb{H}^N )}\leq c$ then \eqref{4NLSeq} is globally well-posed, i.e. $\psi(t,\cdot)$ exists for all $t\geq 0$. Moreover, if $1<\gamma < 1+\tfrac{8}{N}$, \eqref{4NLSeq} is locally well-posed for arbitrary $L^2$ data, i.e. for any $\psi_0\in L^2 (\mathbb{H}^N)$, there exists $T>0$ such that $\psi(t,\cdot)$ exists for all $|t|< T$.
\end{thm}

\begin{proof}
We begin by proving the global well-posedness for small initial data when $\beta>0$. Let $\psi=\Phi(\varphi)$ be the solution to 
\begin{equation*}
\begin{cases}i\partial_t \psi(t,x) +P \psi(t,x) =F(\varphi(t,x))\mbox{ in}\ \R\times \Hy^N ,\\ \psi(0,\cdot)=\psi_0 \in L^2 (\Hy^N).\end{cases}
\end{equation*}

By Duhamel's formula, we have
$$\psi(t,x)=e^{itP}\psi_0(x) -i\int_0^t e^{i(t-s)P} F(\varphi (s,x))\,ds.$$
Using Strichartz estimate \eqref{stri}, we get
$$\|\psi \|_{L_t^\infty  L_x^2  } +\|\psi \|_{L_t^p  L_x^q } \leq C \|\psi_0\|_{L_x^2 (\Hy^N)} +C \|F(\varphi)\|_{L_t^{\tilde{p}^\prime}  L_x^{\tilde{q}^\prime}},$$
for admissible pairs $(p,q)$ and $(\tilde{p},\tilde{q})$,  i.e.
$$p,\tilde{p}, q,\tilde{q}> 2,\ \dfrac{4}{Np}\geq \dfrac{1}{2} - \dfrac{1}{q},\ \dfrac{4}{N\tilde{p}}\geq \dfrac{1}{2} - \dfrac{1}{\tilde{q}}.$$
By the assumption on $F$, we have
$$\|F(\varphi )\|_{L_t^{\tilde{p}^\prime} L_x^{\tilde{q}^\prime}} \leq C \|\varphi\|^\gamma_{L_t^{\gamma\tilde{p}^\prime} L_x^{\gamma\tilde{q}^\prime}}. $$
We impose that $p=\gamma \tilde{p}^\prime$ and $q=\gamma \tilde{q}^\prime$, i.e. $\tilde{p} = \tfrac{p}{p-\gamma}$ and $\tilde{q} = \tfrac{q}{q-\gamma}$. In order for $(p,q)$ and $(\tilde{p},\tilde{q})$ to be admissible, we need that $p>2 $, $q>2$ and
$$\dfrac{1}{2} -\frac{1}{q} \leq \dfrac{4}{Np}\leq \dfrac{1}{\gamma} (\dfrac{1}{2}+\dfrac{4}{N})  -\dfrac{1}{q}. $$
So we can find $p,q,\tilde{p},\tilde{q}$ satisfying all the previous conditions provided that $1<\gamma \leq 1+8/N$. We have proved that $\Phi$ is a self-map of a Banach space $X=C(\R ;L^2 (\Hy^N))\cap L^p (\R ; L^q (\Hy^N))$ 
with the norm
$$\|\psi\|_X =\|\psi\|_{L_t^\infty L_x^2 } +\|\psi\|_{L_t^p L_x^q} .$$ 
We want to show that $\Phi$ is a contraction in the ball
$$X_\varepsilon = \{\psi\in X\colon \|\psi\|_X \leq \varepsilon \},$$
provided that $\varepsilon>0$ and $\|\psi_0\|_{L^2 (\Hy^N)}$ are small enough. Let $\varphi,\tilde{\varphi} \in X$. We set $\psi=\Phi(\varphi)$ and $\tilde{\psi}=\Phi (\tilde{\varphi})$. Using \eqref{stri} and our assumption on $F$, we obtain
$$\|\psi-\tilde{\psi}\|_X \leq C  \|F(\varphi)- F(\tilde{\varphi})\|_{L_t^{\tilde{p}^\prime} L_x^{\tilde{q}^\prime}}\leq C \|\varphi-\tilde{\varphi}\|_{L_t^p L_x^q} \big(\|\varphi\|^{\gamma-1}_{L_t^p L_x^q}+\|\tilde{\varphi}\|^{\gamma-1}_{L_t^p L_x^q} \big).$$
So, we deduce that
$$\|\psi-\tilde{\psi}\|_X \leq C \|\varphi-\tilde{\varphi}\|_X \big(\|\varphi\|_X^{\gamma -1} +\|\tilde{\varphi}\|_X^{\gamma -1} \big).$$
If we assume that $\|\varphi\|_X , \|\tilde{\varphi}\|_X \leq \varepsilon$ and $\|\psi_0\|_{L^2 (\Hy^N)}\leq \delta$, then we get
$$\|\psi\|_X , \|\tilde{\psi}\|_X\leq C\delta +C \varepsilon^\gamma,\quad\text{and}\quad \|\psi-\tilde{\psi}\|_X \leq 2C \varepsilon^{\gamma -1} \|\varphi- \tilde{\varphi}\|_X .$$
Therefore, choosing $C \varepsilon^{\gamma -1}\leq \tfrac{1}{4}$ and $C\delta \leq \tfrac{3}{4} \varepsilon$, we obtain
$$\|\psi\|_X, \|\tilde{\psi}\|_X \leq \varepsilon ,\quad\text{and}\quad \|\psi-\tilde{\psi}\|_X \leq \frac{1}{2} \|\varphi-\tilde{\varphi}\|_X .$$
This concludes the proof of the global well-posedness for small $L^2$ initial data. Next, we deal with an arbitrary initial data in $L^2$.

When $1<\gamma<1+\tfrac{8}{N}$, we restrict to a small time interval $I=[-T,T]$. We proceed as above except that we define $(1-\tilde{\lambda}) p=\gamma \tilde{p}^\prime$, where $\tilde{\lambda}>0$ is small enough. So, we get
$$\|F(\varphi)\|_{L_t^{\tilde{p}^\prime} L_x^{\tilde{q}^\prime}} \leq C \|\varphi\|^\gamma_{L_t^{(1-\tilde{\lambda}) p} L_x^{q}}. $$ 
Using H\"older's inequality in time, we obtain
$$\|\psi\|_X \leq C \|\psi_0\|_{L^2(\Hy^N)} +CT^\lambda \|\varphi\|_X^\gamma,$$
for some $\lambda >0$, and
$$\|\psi-\tilde{\psi}\|_X \leq C T^\lambda \big(\|\varphi\|_X^{\gamma -1}+ \|\tilde{\varphi}\|_X^{\gamma -1}\big) \|\varphi-\tilde{\varphi}\|_X.$$
So $\Phi$ is a contraction in the ball
$$X_R = \{\psi\in X \colon \|\psi\|_X \leq R  \},$$
if $R$ is large enough and $T$ small enough. 
\end{proof}

\begin{thm}
\label{thmexh2}
Let $1<\gamma \leq \tfrac{N+4}{(N-4)_+}$ and $\beta >0$, then \eqref{4NLSeq} is globally well-posed for small $H^2$ data. Moreover, if $1<\gamma <\tfrac{N+4}{(N-4)_+}$, then \eqref{4NLSeq} is locally well-posed for arbitrary $H^2$ data.
\end{thm}

\begin{proof}
As previously, we begin by showing the global well-posedness for small initial data when $\beta >0$. We apply $-\Delta_{\Hy^N}$ to our equation. Then the Strichartz estimate gives, for any admissible pairs $(p,q)$ and $(\tilde{p},\tilde{q})$,
$$\|\psi\|_{L_t^\infty H_x^2} +\|\psi \|_{L_t^p H_x^{2,q}}\leq C \|\psi_0 \|_{H^2_x (\Hy^N)}+C \|F(\varphi)\|_{L_t^{\tilde{p}^\prime} H_x^{2,\tilde{q}^\prime}}.  $$
On the other hand, we have
$$\|\Delta F(\varphi) \|_{L_t^{\tilde{p}^\prime} L_x^{\tilde{q}^\prime}} \leq C \|\varphi \|^\gamma_{L_t^p H_x^{2,q}},$$
and, with the Sobolev inequality,
$$ \| F(\varphi) \|_{L_t^{\tilde{p}^\prime} L_x^{\tilde{q}^\prime}} \leq C \|\varphi \|^\gamma_{L_t^p H_x^{2,q}},$$
provided that $p=\gamma \tilde{p}^\prime$ and $\frac{1}{\tilde{q}^\prime}\geq \frac{\gamma}{q} - \frac{2(\gamma -1)}{N}$. Combining the previous inequalities, we obtain that
$$\|\psi\|_{L_t^\infty H_x^2} +\|\psi \|_{L_t^p H_x^{2,q}}\leq C \|\psi_0\|_{H^2_x (\Hy^N)}+C \|\varphi\|^\gamma_{L_t^p H_x^{2,q}}.$$
We can choose $p,\tilde{p},q,\tilde{q}$ as above since
$$\dfrac{1}{2}- \dfrac{1}{q} \leq \dfrac{4}{Np}\leq \dfrac{1}{\gamma} (\dfrac{4}{N}+\dfrac{1}{2}) - \dfrac{1}{q} +\dfrac{2(\gamma -1)}{\gamma N}, $$
provided that $\gamma \leq \tfrac{N+4}{(N-4)_+}$. So we proved that $\Phi$ is a self-map of the Banach space $X=C(\R ;H^2 (\Hy^N))\cap L^p (\R ; H^{2,q} (\Hy^N))$ 
with the norm 
$$\|\psi\|_X =\|\psi\|_{L_t^\infty H_x^2 } +\|\psi\|_{L_t^p H_x^{2,q}} .$$ 
We want to show that $\Phi$ is a contraction in a ball
$$X_\varepsilon = \{\psi\in X\colon \|\psi\|_X \leq \varepsilon \},$$
provided that $\varepsilon>0$ and $\|\psi_0\|_{H^2 (\Hy^N)}$ are small enough. Let $\varphi,\tilde{\varphi} \in X_\varepsilon$. We set $\psi=\Phi(\varphi)$ and $\tilde{\psi}=\Phi (\tilde{\varphi})$. As previously, we obtain
\begin{align*}
\|\psi-\tilde{\psi}\|_X &\leq C  \|F(\varphi)- F(\tilde{\varphi})\|_{L_t^{\tilde{p}^\prime} L_x^{\tilde{q}^\prime}}\\
&\leq C \|\varphi-\tilde{\varphi}\|_{L_t^p H_x^{2,q}} \big(\|\varphi\|^{\gamma-1}_{L_t^p H_x^{2,q}}+\|\tilde{\varphi}\|^{\gamma-1}_{L_t^p L_x^q} \big)\\
&\leq 2C \varepsilon^{\gamma -1} \|\varphi-\tilde{\varphi}\|_{L_t^p H_x^{2,q}}.
\end{align*}
So, if $\varepsilon$ is small enough, $\Phi$ is a contraction in $X_\varepsilon$ for the norm inherited from $Y=L^p (\R; L^q (\Hy^N)$. We deduce from this fact the uniqueness of the possible fixed point of $\Phi$ in $X_\varepsilon$. Concerning the existence, let $\psi_0 \in X_\varepsilon$. We define $\psi_j=\Phi^j (\psi_0)$. We see that $\psi_j$ converges to some fixed point $\psi$ in the closure of $X_\varepsilon$ in $Y$. On the other hand, since $X$ is reflexive and separable, $\psi_j$ weakly converges in $X_\varepsilon$ to some $\tilde{\psi}$. But by the uniqueness, we deduce that $\psi=\tilde{\psi}$. This concludes the proof of the global well-posedness for small $H^2$ initial data. Concerning the local well-posedness for arbitrary initial data in $H^2$ when $1<\gamma<\frac{N+4}{(N-4)_+}$, we restrict to a small time interval $I=[-T,T]$. We proceed as previously, using that
$$\|\psi\|_X \leq C \|\psi_0\|_{H^2 (\Hy^N)} +CT^\lambda \|\varphi\|_X^\gamma,$$
for some $\lambda >0$, and
$$\|\psi-\tilde{\psi}\|_Y \leq C T^\lambda \big(\|\varphi\|_X^{\gamma -1}+ \|\tilde{\varphi}\|_X^{\gamma -1}\big) \|\varphi-\tilde{\varphi}\|_Y,$$
where $X= C(I; H^2 (\Hy^N)) \cap L^p (I; H^{2,q} (\Hy^N))$ and $Y=L^p (I; L^q (\Hy^N))$.

\end{proof}

Another consequence of our Strichartz estimate are the following scattering results in $L^2 (\Hy^N)$ and $H^2 (\Hy^N)$.

\begin{thm}
Let 
$\beta >0$. Assume that $1<\gamma \leq 1+\tfrac{8}{N}$. Let $\psi$ be a global solution with small $L^2$ data of \eqref{4NLSeq} obtained in Theorem \ref{thmexl2}. Then $\psi$ has the following scattering property: there exist $\psi_\pm \in L^2 (\Hy^N)$ such that
$$\|\psi(t)- e^{itP}\psi_\pm \|_{L^2 (\Hy^N)} \rightarrow 0\ as\ t\rightarrow \pm \infty.$$
\end{thm}

\begin{proof}
The scattering follows from the following Cauchy criterion: if 
\[
\|z(t_1 ,\cdot)-z(t_2 ,\cdot)\|_{L^2 (\Hy^N)} \rightarrow 0
\]
as $t_i \rightarrow \infty$,
 then there exists $z_+ \in L^2 (\Hy^N)$ such that $\|z(t,\cdot)-z_+(\cdot)\|_{L^2 (\Hy^N)}\rightarrow 0$ as $t\rightarrow \infty$.
We apply this criterion for $z(t,x)=e^{-itP}\psi (t,x)$. Using the Strichartz estimate, we have
\begin{align*}\Big\|e^{-it_2 P} \psi (t_2 ,\cdot) - e^{-it_1 P}\psi (t_1 ,\cdot)\Big\|_{L^2 (\Hy^N)} &= \Big\|\int_{t_1}^{t_2} e^{-is P} F(\psi(s,\cdot))\, ds\Big\|_{L^2 (\Hy^N)}\\
&\leq \big\|\psi\big\|_{L^p ([t_1,t_2] ;L^q (\Hy^n))}^\gamma .
\end{align*}
Since $u\in L^p (\R ;L^q (\Hy^n))$, the last term vanishes when $t_1\leq t_2$ tend both to $\pm\infty$.
\end{proof}
With the same proof but using Theorem \ref{thmexh2} instead of Theorem \ref{thmexl2}, we obtain the following:
\begin{thm}
Assume that $1<\gamma \leq \frac{N+4}{(N-4)_+}$. Let $\psi$ be a global solution with small $H^2$ data of \eqref{4NLSeq} obtained in Theorem \ref{thmexh2}. Then $\psi$ has the following scattering property: there exist $\psi_\pm \in H^2 (\Hy^N)$ such that
$$\|\psi(t)- e^{itP}\psi_\pm \|_{H^2 (\Hy^N)} \rightarrow 0\ as\ t\rightarrow \pm \infty.$$
\end{thm}

We can also prove the existence of wave operator following Banica-Carles-Staffilani \cite{BaCaSt}.

\begin{thm}
Assume that $F(z)=|z|^{\gamma -1} z$, $\beta>0$ and that $\gamma <\tfrac{N+4}{(N-4)_+}$. Then, for any $\psi_0\in H^2 (\Hy^N)$, \eqref{4NLSeq} has a unique global solution $\psi(t,\cdot)$ with the scattering property 
$$\|\psi(t) - e^{itP}\psi_0\|_{H^2 (\Hy^N)} \rightarrow 0\ as\ t\rightarrow \pm \infty .$$
\end{thm}

Finally, we conclude this section by proving a rough trapping result. Before stating our result, we need some notation. We set
$$\|u\|^2_{H_{\beta,\lambda} (\Hy^N)} =\int_{\Hy^N} \big( |\Delta_{\Hy^N}  u|^2 +\beta |\nabla u|^2 +\lambda u^2 \big)\, dV. $$
We recall the following higher order Poincar\' e inequality (see Berchio-Ganguly \cite{BeGa}) for $k,j\in \N$ with $k>j$,
$$\int_{\Hy^N} |\nabla^k u |^2 \,dV \geq (\dfrac{N-1}{2})^{2(k-l)} \int_{\Hy^N} |\nabla^l u|^2 \,dV,$$
where 
$$\nabla^j =\begin{cases} \Delta_{\Hy^N}^{j/2},&\mbox{ if $j$ is even},\\ \nabla (\Delta_{\Hy^N}^{(j-1)/2}),&\mbox{ if $j$ is odd}.  \end{cases}$$
So assuming that $\beta\geq 0$, we see that $\| \cdot\|_{H_{\beta,\lambda}(\Hy^N)}$ is equivalent to $\|\cdot\|_{H^2 (\Hy^N)}$ if 
\begin{equation}
\label{lambdacoer}
\lambda >-\left(\frac{N-1}{2}\right)^4 -\beta \left(\frac{N-1}{2}\right)^2.
\end{equation}
We set
$$E_{\beta,\lambda}(u)= \frac{1}{2}\|u\|^2_{H_{\beta,\lambda} (\Hy^N)}- \frac{1}{2\sigma +2}\|u \|_{L^{2\sigma +2}(\Hy^N)}^{2\sigma +2}.$$
If $\beta \geq 0$, $\lambda$ satisfies \eqref{lambdacoer}, and $0<\sigma<\frac{4}{(N-4)_+}$, it is standard to prove that the minimisation problem
\begin{equation}
\label{min}
D^{-1}=\min_{u\in H^2 (\Hy^N)}\dfrac{\|u\|_{H_{\beta,\lambda} (\Hy^N) }^2 }{\|u\|_{L^{2\sigma +2}(\Hy^N) }^{2} },
\end{equation}
admits a solution $Q$. We will say that $Q$ is a ground state. We denote by $\mathcal{Q}$ the set of all ground states, i.e.
$$\mathcal{Q}=\{Q\in H^2 (\Hy^N)\colon \ Q\ \text{is a  solution to \eqref{min}}\}.$$ 
 It is easy to check that $Q$ satisfies
\begin{equation}
\label{eqell}
\Delta_{\Hy^N}^2 Q -\beta \Delta_{\Hy^N} Q +\lambda Q= D^{-1} |Q|^{2\sigma} Q.
\end{equation}
So, setting $\tilde{Q}= D^{-(2\sigma +1)} Q$, we see that $\tilde{Q}$ satisfies \eqref{eqell} with $D^{-1}\equiv 1$. 
We set $$\delta_\lambda (u)= \|u\|_{H_{\beta,\lambda} (\Hy^N)}^2 - \|Q\|_{H_{\beta,\lambda} (\Hy^N)}^2,$$
 where $Q \in \mathcal{Q}$. We prove the following:

\begin{thm}
\label{trapping}
Let $N\geq 2$, $\beta\geq 0$, $\lambda$ satisfying \eqref{lambdacoer} and $\psi (t)$ be the solution to 
$$\begin{cases}
i\partial_t\psi = -\Delta_{\Hy^N}^2\psi +\beta\Delta_{\Hy^N} \psi +|\psi|^{2\sigma}\psi \quad\text{in }\R\times \Hy^N,\\
\psi(0,\cdot)=\psi_0\in H^2 (\Hy^N).
\end{cases}$$
  If $0<\sigma<\frac{4}{(N-4)_+}$ and $E_{\beta ,\lambda} (\psi_0)\leq E_{\beta,\lambda} (Q)$, $Q\in \mathcal{Q}$, then $\delta_\lambda (\psi (t))$ does not change sign. Moreover,
\begin{itemize}
\item If $\delta_\lambda (\psi_0)=0$, then there exists $\theta \in \R$ and a hyperbolic isometry $h$ such that $\psi_0=e^{i\theta}Q\big(h(\cdot)\big)$ , $Q\in \mathcal{Q}$.
\item if $\delta_\lambda (\psi_0)<0$, then the solution $\psi$ is global in time.
\item if $\delta_\lambda (\psi_0)>0$ then the solution $\psi$ does not scatter in any time direction.
\end{itemize}
\end{thm}

\begin{proof}
We begin by proving the following claim: if $E_{\beta,\lambda} (\psi)\leq E_{\beta,\lambda} (Q)$ and 
\[
\|\psi\|_{H_{\beta,\lambda} (\Hy^N)}^2  \leq \|Q\|_{H_{\beta,\lambda} (\Hy^N)}^2,
\] 
then there holds
$$\|\psi\|_{H_{\beta,\lambda} (\Hy^N)}^2  \leq \dfrac{E_{\beta,\lambda} (\psi)}{E_{\beta,\lambda} (Q)}\|Q\|_{H_{\beta,\lambda} (\Hy^N)}^2 .$$
Indeed, first as noticed previously, $Q$ satisfies \eqref{eqell} with $D=1$. So, we find 
$$\|Q\|_{H_{\beta,\lambda} (\Hy^N)}^2 =\|Q\|_{L^{2\sigma +2} (\Hy^N)}^{2\sigma +2}.$$
This implies that
$$D=\|Q\|_{L^{2\sigma+2} (\Hy^N)}^{-2\sigma}=\|Q\|_{H_{\beta,\lambda} (\Hy^N)}^{\frac{-4\sigma}{2\sigma +2}} ,$$
and
\begin{equation}
\label{etrapping1}
E_{\beta,\lambda} (Q)= \frac{\sigma}{2\sigma +2}\|Q\|_{H_{\beta,\lambda} (\Hy^N)}^2 .
\end{equation}
So, using that $Q$ is a ground state, we get 
\begin{align*}
E_{\beta,\lambda} (\psi)&= \frac{1}{2}\|\psi\|_{H_{\beta,\lambda} (\Hy^N)}^2 - \frac{1}{2\sigma +2} \|\psi\|_{L^{2\sigma +2}(\Hy^N)}^{2\sigma +2}\\
& \geq  \frac{1}{2}\|\psi\|_{H_{\beta,\lambda} (\Hy^N)}^2 - \frac{D^{\sigma +1}}{2\sigma +2} \|\psi\|_{H_{\beta,\lambda} (\Hy^N)}^{2\sigma +2}\\
&= a (\|\psi\|_{H_{\beta,\lambda} (\Hy^N)}^2),
\end{align*}
where
$$a(x)=\dfrac{1}{2}x - \dfrac{\|Q\|_{H_{\beta,\lambda} (\Hy^N)}^{-2\sigma}}{2\sigma +2} x^{\sigma +1}.$$
Let $b(x)=a(x) - \frac{\sigma}{2\sigma +2} x$. One can check that $b>0$ on $[0,\|Q\|_{H_{\beta,\lambda} (\Hy^N)}^2]$. Since $\|\psi\|_{H_{\beta,\lambda} (\Hy^N)}^2 \leq \|Q\|_{H_{\beta,\lambda} (\Hy^N)}^2$, we get
$$E_{\beta,\lambda}(\psi)\geq \dfrac{\sigma}{2\sigma +2}\|\psi\|_{H_{\beta,\lambda} (\Hy^N)}^2.$$
Using \eqref{etrapping1}, this proves the claim.

Now suppose that $\delta (\psi_0)=0$. Then $E_{\beta,\lambda}(\psi_0)=E_{\beta,\lambda}(Q)$. So $\psi_0$ is a ground state. As a consequence, if $\delta (\psi_0)\neq 0$, then $\delta (\psi(t))\neq 0$ for all $t$ such that $\psi(t)$ exists.

Next, assume that $\delta (\psi_0)<0$. Since the sign of $\delta (\psi(t))$ does not depend on $t$, we deduce that $\|\psi(t)\|_{H_{\beta,\lambda} (\Hy^N)}\leq C$, for some constant $C$ not depending on $t$. Therefore, using the mass conservation, $\psi(t)$ is uniformly bounded in $H^2 (\Hy^N)$ so it exists globally.

Finally, suppose that $\delta (\psi(t))>0$ and that $\psi(t)$ scatters. In particular, we have
$$\lim_{t\rightarrow \infty} \|\psi\|_{L^{2\sigma +2}(\Hy^N)}=0.$$ 
So, for $\varepsilon>0$, we can find $t$ large enough such that
$$ \frac{1}{2} \|\psi(t)\|_{H_{\beta,\lambda} (\Hy^N)}^2 < E_{\beta,\lambda} (\psi (t)) +\varepsilon = E_{\beta,\lambda} (\psi_0) +\varepsilon. $$
On the other hand, we have
$$E_{\beta,\lambda}(\psi_0)\leq E_{\beta,\lambda}(Q) < \frac{1}{2} \|\psi (t)\|_{H_{\beta,\lambda} (\Hy^N)}^2 -\frac{1}{2\sigma +2} \|Q\|_{L^{2\sigma +2}(\Hy^N)}^{2\sigma +2}.$$
So, taking $\varepsilon=\frac{1}{4(\sigma +1)} \|Q\|_{L^{2\sigma +2}(\Hy^N)}^{2\sigma +2} $, we get a contradiction.
\end{proof}

\section{Strichartz estimates for radial solutions on a complete rotationally symmetric Riemannian manifold}
\label{secstri2}
In this section we adapt the method from \cite{BaDyweight} to obtain Strichartz estimates on rotationally symmetric manifolds. We assume that $M=(\R^N,g)$ is an $N$-dimensional complete rotationally symmetric 
manifold with the Riemannian metric 
$g=dr^2 + \phi^2(r)d\theta^2$, where $d\theta^2$ is the standard metric on the sphere $\Sph^{N-1}$ and 
$\phi$ is a $C^\infty$-smooth nonnegative function on $[0,\infty)$, strictly positive on $(0,\infty)$ such that 
$\phi^{(\text{2k})}(0)=0\,\ k=0,1,\ldots,$ and $\phi^\prime(0)=1$.
We recall that the Laplace-Beltrami operator on $M$ is  given by
\[
\Delta_{M}=\partial_r^2 +(N-1)\frac{\phi^\prime (r)}{\phi(r)}\partial_r
+\frac{1}{\phi^2(r)}\Delta_{\Sph^{N-1}},
\]
where $\Delta_{\Sph^{N-1}}$ is the Laplace-Beltrami operator on the unit sphere $\Sph^{N-1}$.
As in \cite[\S 2.1]{BaDyweight} we set $\sigma(r)=\big(r/\phi(r)\big)^{(N-1)/2}$ and 
$\psi(t,r)=\sigma (r)\varphi(t,r)$. Then the Laplacian of a radial function $\psi$ on $M$ can be expressed in terms of the Euclidean Laplacian, denoted by $\Delta_{\R^N}$ in what follows, with a potential as 
$$\Delta_M \psi = \sigma (\Delta_{\R^N} \varphi +V\varphi ),$$
where 
\begin{align*}
V &=\frac{\sigma''}{\sigma}+(N-1)\frac{\phi'}{\phi}\frac{\sigma'}{\sigma}\\
&=\frac{N-1}{2}\left[\frac{N-3}{2}\left(\frac{1}{r^2}
-\left(\frac{\phi^\prime}{\phi}\right)^2\right)
-\frac{\phi''}{\phi}\right].
\end{align*}
By noticing that 
\[
2\frac{\sigma'}{\sigma}+ (N-1) \frac{\phi'}{\phi}=\frac{N-1}{r},
\]
we can write the bi-Laplacian on $M$ as
\begin{align*}
\Delta^2_M \psi &= \sigma \big(\Delta_{\R^N}^2 \varphi +2V \Delta_{\R^N} \varphi +2\varphi ' V' +\varphi (\Delta_{\R^N} V +V^2)\big)\\
&=\sigma (-\Delta_{\R^N} -V ) (-\Delta_{\R^N} -V) \varphi.
\end{align*}
Let $\tilde{P}_M \psi= \Delta^2_M \psi -\beta \Delta_M \psi $. Then we have
\begin{align*}
\tilde{P}_M \psi &= \sigma \big[ (-\Delta_{\R^N} -V ) (-\Delta_{\R^N} -V) +\beta (-\Delta_{\R^N} -V) \big] \varphi\\
&=\sigma  (-\Delta_{\R^N} - V) (-\Delta_{\R^N} -V +\beta) \varphi.
\end{align*}
We denote $P_V= \tilde{P}_M /\sigma$. Observe that, for any $\mu \in \mathbb{C}$, 
$$P_V- \mu= (-\Delta_{\R^N} -V +\gamma_1) (-\Delta_{\R^N} -V +\gamma_2),$$
where $\gamma_1 \gamma_2= -\mu$, $\gamma_1 +\gamma_2 = \beta$. Thanks to this decomposition, we can split the resolvent $R(P_V-\mu)=(P_V-\mu)^{-1}$ into a sum of resolvents of two second order operators. Indeed, we have
\begin{equation}
\label{decom}
R(P_V-\mu) =\dfrac{ (-\Delta_{\R^N} -V +\gamma_1)^{-1} - (-\Delta_{\R^N} -V + \gamma_2)^{-1}}{\gamma_2 -\gamma_1} .
\end{equation}
Next, let us recall the resolvent estimate of 
Burq et al. \cite[Theorem 2.1]{BPST}; see  \cite[Theorem 3.1]{BaDyweight}.
\begin{thm} \label{thmbd}
Let $V\in C^1 (\R^N \setminus \{0\})$ such that
$$|V(x)|\leq C |x|^{-2}.$$
Suppose that there exists $\delta_0>0$ such that
\[
\left(\frac{N}{2} -1\right)^2 +r^2 V(r) \geq \delta_0
\]
and
\[
\left(\frac{N}{2} -1\right)^2 - r^2 \partial_r (rV)(r) \geq \delta_0
\]
in $\R^N \setminus \{0\}$.
Then there exists $C>0$ such that
$$\sup_{\mu \in \mathbb{C} \setminus \R } \left\| |x|^{-1} (-\Delta_{\R^N} +V -\mu)^{-1} |x|^{-1}\right\|_{L^2 (\R^N) \rightarrow L^2 (\R^N)}\leq C.$$
\end{thm}

Thanks to the previous theorem, we will prove a  resolvent estimate for our operator.

\begin{prop}
Assume that $V$ satisfies the assumption in Theorem \ref{thmbd}. Moreover, assume that
\begin{equation}
\label{hyppot}
\left|\Delta_{\R^N} \frac{V}{\langle x\rangle}\right| +\left|\nabla\frac{V'}{\langle x\rangle}\right|+\dfrac{ \big| \Delta_{\R^N} V +V^2 +\mu \big|}{\langle x \rangle} + \frac{|V|}{\langle x\rangle} \leq \frac{C}{\langle x \rangle},
\end{equation}
where $\langle x\rangle =(1+|x|^2)^{1/2}$.
Then there exists $C>0$ such that
$$\sup_{\mu \in \mathbb{C} \setminus \R} \big\| \langle x\rangle^{-1}  (P_V -\mu)^{-1} \langle x\rangle ^{-1} \big\|_{L^2 (\R^N) \rightarrow H^2 (\R^N)}\leq C. $$
\end{prop}

\begin{proof}
First, proceeding as in \cite{BaDyweight} (see in particular \cite[(57)]{BaDyweight}), we see thanks to Theorem \ref{thmbd} and  \eqref{decom} that 
$$\sup_{\mu \in \mathbb{C} \setminus \R} \big\| \langle x\rangle ^{-1}  (P_V -\mu)^{-1} \langle x\rangle ^{-1} \big\|_{L^2 \rightarrow H^1}\leq C. $$
Let $(P_V-\mu) u= f/\langle x\rangle$.
 Multiplying by $\bar{u}/\langle x\rangle$ and integrating by parts, we find
\begin{align*}
\int &\frac{|\Delta_{\R^N} u|^2}{\langle x\rangle}\,dx \leq C \int \left(\frac{|\nabla u|^2}{\langle x\rangle^3} + \frac{|u|^2}{\langle x\rangle^5} + \frac{f\bar{u}}{\langle x\rangle^2}\right)\,dx 
+  \int \bigg(\frac{(2V-\beta) |\nabla u|^2}{\langle x\rangle}
\\
& 
+ |u|^2 \left(\frac{1}{2} \Delta_{\R^N}\frac{(2V-\beta)}{\langle x\rangle} + \frac{\Delta_{\R^N} V}{\langle x \rangle} - \nabla V \cdot\nabla \frac{1}{\langle x\rangle}+\frac{\Delta_{\R^N} V +V^2 - \beta V +\mu }{\langle x\rangle} \right)\bigg)\,dx.
\end{align*}

So, if 
\eqref{hyppot} holds, 
then using once more Theorem \ref{thmbd}, we deduce that
$$\sup_{\mu \in \mathbb{C} \setminus \R}\big\|\langle x\rangle ^{-1} (P_V-\mu)^{-1} \langle x\rangle ^{-1} \big\|_{L^2  \rightarrow H^2 }\leq C.$$
\end{proof}

Now we convert our resolvent estimate into smoothing estimate by using the method of Burq, Gerard and Tzvetkov \cite{BGT}.

\begin{prop}
Let $\varphi$ be a solution to
$$i\partial_t \varphi- P_V \varphi = \tilde{h},\ \varphi(0)=\varphi_0.$$
Assume that $V$ satisfies the assumptions of Theorem \ref{thmbd} and \eqref{hyppot}. Then there exists a constant $C>0$ such that for all $\tilde{h}$ with $\langle x\rangle \tilde{h} \in L^2 (\R , L^2)$, we have
$$\big\|\langle x \rangle^{-1} \nabla^{i}\varphi\big\|_{L^2 (\R, L^2 (\R^N))} \leq C \big(\|\varphi_0\|_{H^1 (\R^N)} + \|\langle x\rangle  \tilde{h}\|_{L^2 (\R, L^{2}(\R^N))}  \big),$$
for $i=1,2$ with the convention that $\nabla^2 =\Delta_{\R^N}$.
\end{prop}

\begin{proof}
Using the previous proposition and proceeding as in the proof of \cite[Prop. 2.7]{BGT}, 
we find that
$$\|\langle x\rangle ^{-1} \varphi\|_{L^2 (\R, H^2)}\leq C (\|\varphi_0\|_{H^1 (\R^N)} +\|\langle x\rangle  \tilde{h}\|_{L^2 (\R, L^{2})}) .$$
Noticing that
$$\|\langle x\rangle ^{-1} \nabla \varphi\|_{L^2 (\R , L^2)}\leq C\|\langle x\rangle ^{-1} \varphi\|_{L^2 (\R , H^1)} ,$$
and
$$\|\langle x\rangle ^{-1} \Delta_{\R^N} \varphi\|_{L^2 (\R , L^2)}\leq C\|\langle x\rangle ^{-1} \varphi\|_{L^2 (\R , H^2)} ,$$
this establishes the proposition.

\end{proof}

We consider 
$$\begin{cases}
i\partial_t \varphi- \Delta_{\R^N}^2 \varphi = \tilde{h} + 2V\Delta_{\R^N} \varphi +2V' \varphi ' +\varphi \tilde{W}\mbox{ in}\ \R \times \R^N ,\\
u(0,\cdot)=u_0 \in H^1 (\R^N),
\end{cases}$$
where $\tilde{W}=\Delta_{\R^N} V +V^2$. Using our smoothing estimate and the Strichartz estimate for the biharmonic Schr\" odinger equation, we will derive a Strichartz estimate for our equation with potentials. We will need the following two sets of assumptions: if we consider an $S$-admissible pair, we assume that
\begin{equation}\label{V-conds-S}
\|\langle x\rangle \tilde{W}\|_{L^{N,\infty}(\R^N)},\  \|\langle x\rangle V^\prime\|_{L^{N,\infty}(\R^N)},\ \|\langle x\rangle V\|_{L^{N,\infty}(\R^N)}\leq C,
\end{equation}
whereas, for a $B$-admissible pair, we assume that
\begin{equation}\label{V-conds-B}
\|\langle x\rangle^2 \tilde{W}\|_{L^{N/2,\infty}(\R^N)},\  \|\langle x\rangle^2 V^\prime\|_{L^{N/2,\infty}(\R^N)},\ \|\langle x\rangle^2 V\|_{L^{N/2,\infty}(\R^N)}\leq C.
\end{equation}

\begin{thm}
Let $\varphi$ be a solution to
$$\begin{cases}
i\partial_t \varphi- P_V  \varphi =\tilde{h}\mbox{ in}\ \R \times \R^N ,\\
u(0,\cdot)=u_0 \in H^1 (\R^N).
\end{cases}$$
Assume that $V$ satisfies the assumptions of Theorem \ref{thmbd}, \eqref{hyppot} and \eqref{V-conds-S} (resp. \eqref{V-conds-B}). Then, there exists a constant $C>0$ such that for all $B$- (resp. $S$-) admissible pairs $(p_i,q_i)$, $i=1,2$,
 we have
$$\|u\|_{L^{p_1}(\R,L^{q_1} (\R^N))} \leq C \big(\|u_0\|_{H^1 (\R^N)} +\|\tilde{h}\|_{L^{p_2^\prime} (\R , L^{q_2^\prime} (\R^N))}\big).$$
\end{thm}

\begin{proof}
We start by considering the $B$-admissible pairs case.
Using the standard end point Strichartz estimate in Lorentz space, we have
\begin{equation}
\label{stripausa}
\|\varphi\|_{L^2 (\R, L^{\frac{2N}{(N-4)_+},2 })} \leq \|\varphi_0\|_{L^2 (\R^N)} + \| \tilde{h} + 2V\Delta_{\R^N} \varphi +2V' \varphi' +\varphi\tilde{W}\|_{L^2 (\R; L^{\frac{2N}{N+4} ,2})}. 
\end{equation}
We recall the H\"older's inequality in Lorentz space: let $0<p_1,p_2,p<\infty$ and  $0<q_1,q_2, q \leq \infty$ such that  $\frac{1}{p}= \frac{1}{p_1}+\frac{1}{p_2}$ and  $\frac{1}{q}=\frac{1}{q_1}+\frac{1}{q_2}$, then
$$\|fg\|_{L^{p,q}} \leq \|f\|_{L^{p_1,q_1}}\|g\|_{L^{p_2 , q_2}}.$$
So, we have, using the previous proposition and 
\eqref{V-conds-B},  
that
$$ \| \tilde{h} \|_{L^2 (\R ; L^{\frac{2N}{N+4} ,2})} \leq \|\langle x\rangle ^{-2}\|_{L^{N/2,\infty}} \|\langle x\rangle^2  \tilde{h}\|_{L^2 (\R ; L^2)}\leq  \|\langle x\rangle^2  \tilde{h}\|_{L^2 (\R ; L^2)},$$
$$ \| \varphi\tilde{W} \|_{L^2 (\R ; L^{\frac{2N}{N+4} ,2})} \leq \|\langle x\rangle^2 \tilde{W}\|_{L^{N/2,\infty}} \|\langle x\rangle ^{-2} \varphi\|_{L^2 (\R ; L^2)}\leq C \big(\|\varphi_0\|_{L^2 } + \|\langle x\rangle^2  \tilde{h}\|_{L^2 (\R ; L^2)}\big), $$
$$ \| V' \varphi' \|_{L^2 (\R ; L^{\frac{2N}{N+4} ,2})} \leq \|\langle x\rangle^2 V' \|_{L^{N/2,\infty}} \|\langle x\rangle ^{-2} \varphi' \|_{L^2 (\R ; L^2)}\leq C \big(\|\varphi_0\|_{L^2} + \|\langle x\rangle^2  \tilde{h}\|_{L^2 (\R ; L^2)}\big) ,$$
and
\begin{align*}
 \| V\Delta_{\R^N} \varphi \|_{L^2 (\R ; L^{\frac{2N}{N+4} ,2})}& \leq \|\langle x\rangle^2 V \|_{L^{N/2,\infty}} \|\langle x\rangle ^{-2} \Delta_{\R^N} \varphi \|_{L^2 (\R ; L^2)} \\
&\leq C  \big(\|u_0\|_{H^1} + \|\langle x\rangle^2  \tilde{h}\|_{L^2 (\R ; L^2)}\big) .
\end{align*}
Therefore, we get that
\begin{equation}
\label{endpointstr}
\|\varphi\|_{L^2 (\R , L^{\frac{2N}{(N-4)_+},2 })} \leq  C  \big(\|\varphi_0\|_{H^1} + \|\langle x\rangle^2  \tilde{h}\|_{L^2 (\R ; L^2)}\big) .
\end{equation}
Let $$A(f)(t,x)=i\int_0^t e^{i(t-\tau)P} f(\tau,x)\, d\tau.$$
Let $g\in C_0^\infty (\R \times \R^N)$ and  $T>0$ such that $\supp\ g \subset (-T,T) \times \R^N$. Then
$$A^\ast (g)(t,x)=i\int_t^T e^{i(t-\tau)P}g(\tau ,x)\, d\tau .$$
So, we get from \eqref{endpointstr} that
$$\|A^\ast (g)\|_{L^2 (\R, L^{\frac{2N}{(N-4)_+},2 })} \leq C \|\langle x\rangle^2  g\|_{L^2 (\R ; L^2)}.$$
In fact, we have $ \|\langle x\rangle ^{-1} \varphi\|_{L^2 (\R; L^2)} \leq  \|\langle x\rangle  \tilde{h}\|_{L^2 (\R ; H^{-2})} $ (this is the dual of the previous proposition with $\psi(0)\equiv 0$) which implies that
$$\|A^\ast (g)\|_{L^2 (\R, L^{\frac{2N}{(N-4)_+},2 })} \leq C \|\langle x\rangle^2  g\|_{L^2 (\R; H^{-2})}.$$
By duality,
$$\|\langle x\rangle ^{-2}A(f) \|_{L^2 (\R; H^2)} \leq C \|f\|_{L^2 (\R , L^{\frac{2N}{N+4},2 })}. $$
We deduce from this that, for $i=1,2$, 
$$\|\langle x\rangle ^{-2} \nabla^{i} A(f) \|_{L^2 (\R; L^2)} \leq C \|f\|_{L^2 (\R , L^{\frac{2N}{N+4},2 })}. $$
This implies that
\begin{equation}
\label{CHstri1}
\|\varphi\|_{L^2 (\R , L^{\frac{2N}{(N-4)_+},2 })} \leq C\big( \|\varphi_0\|_{H^1} +   \|\tilde{h}\|_{L^2 (\R , L^{\frac{2N}{N+4},2 })}\big).
\end{equation}
Next, using $\dfrac{d}{dt}\int |\varphi|^2 dx = \im \int \tilde{h} \bar{\varphi}\, dx$, we get that
\begin{equation}
\label{CHstri2}
\|\varphi\|_{L^\infty (\R ,L^2)} \leq  C\big( \|\varphi_0\|_{L^2} +   \|\tilde{h}\|_{L^1 (\R, L^{2}}\big), 
\end{equation}
and
\begin{equation}
\label{CHstri3}
\|\varphi\|_{L^\infty (\R ,L^2)} \leq  C\big( \|\varphi_0\|_{H^1} +    \|\tilde{h}\|_{L^2 (\R, L^{\frac{2N}{N+4} })} \big).
\end{equation}
Interpolating \eqref{CHstri1}, \eqref{CHstri2}, \eqref{CHstri3} as well as their dual, we get the result for $B$-admissible pairs. The result for $S$-admissible pairs follows in a similar way but instead of \eqref{stripausa}, we use 
$$\|\varphi\|_{L^2 (\R, L^{\frac{2N}{N-2},2 })} \leq \|\varphi_0\|_{L^2} + \| \tilde{h} + 2V\Delta_{\R^N} \varphi +2V' \varphi ' +\varphi \tilde{W}\|_{L^2 (\R; L^{\frac{2N}{N+2} ,2})}. $$
Since $\|\langle x\rangle ^{-1}\|_{L^{N,\infty}}<\infty$, we use H\" older's inequality but this time with weight $\langle x\rangle$ and \eqref{V-conds-S} to find
$$ \| \tilde{h} \|_{L^2 (\R ; L^{\frac{2N}{N+2} ,2})} \leq \|\langle x\rangle ^{-1}\|_{L^{N,\infty}} \|\langle x\rangle  \tilde{h}\|_{L^2 (\R ; L^2)}\leq  \|\langle x\rangle  \tilde{h}\|_{L^2 (\R ; L^2)},$$
$$ \| \varphi\tilde{W} \|_{L^2 (\R ; L^{\frac{2N}{N+2} ,2})} \leq \|\langle x\rangle \tilde{W}\|_{L^{N,\infty}} \|\langle x\rangle ^{-1} \varphi\|_{L^2 (\R ; L^2)}\leq C \big(\|\varphi_0\|_{L^2} + \|\langle x\rangle  \tilde{h}\|_{L^2 (\R ; L^2)}\big), $$
$$ \| V' \varphi ' \|_{L^2 (\R ; L^{\frac{2N}{N+2} ,2})} \leq \|\langle x\rangle V' \|_{L^{N,\infty}} \|\langle x\rangle ^{-1} \varphi ' \|_{L^2 (\R ; L^2)}\leq C \big(\|\varphi_0\|_{L^2} + \|\langle x\rangle  \tilde{h}\|_{L^2 (\R ; L^2)}\big) ,$$
and
\begin{align*}
\| V\Delta_{\R^N} \varphi \|_{L^2 (\R ; L^{\frac{2N}{N+2} ,2})} &\leq \|\langle x\rangle V \|_{L^{N,\infty}} \|\langle x\rangle ^{-1} \Delta_{\R^N} \varphi \|_{L^2 (\R ; L^2)} \\
&\leq C  \big(\|\varphi_0\|_{H^1} + \|\langle x\rangle  \tilde{h}\|_{L^2 (\R ; L^2)}\big) .
\end{align*} 
At this point, we can reproduce the above proof with obvious modifications.
\end{proof}

Finally, we use a change of variables to apply the previous result to our original equation.
\begin{cor}
Assume that $V$ satisfies the assumptions of Theorem \ref{thmbd}, \eqref{hyppot} and \eqref{V-conds-B} (resp. \eqref{V-conds-S}). Let $\psi$ be a radial solution to
$$\begin{cases}
i\partial_t \psi-\Delta_M^2 \psi +\beta \Delta_M \psi= h
\mbox{ in}\ \R \times M,\\
 \psi(0,\cdot)=\psi_0 \in H^1 (M),
\end{cases}$$
 on an interval $I=[0,T]$.
Then, there exists $C>0$ such that, for all interval $I$, and all $B$- (resp. $S$-) admissible pairs $(p_i ,q_i)$, we have
\begin{align*}
&\big\|\psi \sigma^{-(1-2/q_1)}\big\|_{L^{p_1} (I, L^{q_1} (M)) }\\
&\quad \leq C \big(\|\nabla_{M} \psi_0\|_{L^2 (M)} +\|\psi_0 |\nabla_{M} \sigma|/\sigma \|_{L^2 (M)} + \|h\sigma^{(1- 2/q_2)} \|_{L^{p_2^\prime} (I, L^{q_2^\prime } (M))}  \big).
\end{align*}
\end{cor}
\begin{proof}
Using the previous change of variables and the previous theorem, we know that
$$\|\psi /\sigma\|_{L^{p_1} (I, L^{q_1} (\R^N)) } \leq C \big(\|\psi_0/\sigma\|_{H^1 (\R^N)} + \|h /\sigma \|_{L^{p_2^\prime} (I, L^{q_2^\prime } (\R^N))}  \big).$$
The corollary then follows from the remark that
$$\|h/\sigma\|^q_{L^q (\R^N)}=\|h \sigma^{2/q -1}\|^q_{L^q (M)} ,$$
and
\begin{align*}
\|h/\sigma\|^2_{H^1 (\R^N)} & \leq \int_0^\infty \left(\frac{|\nabla h|^2}{\sigma^2} +\frac{h^2 |\nabla \sigma|^2}{\sigma^4}\right) r^{N-1}\,dr \\
& = \int_0^\infty \left(|\nabla h|^2 +h^2 \dfrac{|\nabla \sigma|^2}{\sigma^2}\right)\phi^{N-1}\, dr .
\end{align*}

\end{proof}

\section{Blow-up}
\label{secblowup}
In this section, we work on a complete rotationally symmetric Riemannian manifold $M=(\R^N,g)$ equipped with the metric $g =dr^2 +\phi^2(r) d\theta^2$. 
 We assume that $\phi(r)\ge r$ for every $r\ge 0$. To simplify notation, we denote by $\Delta$ the Laplace-Beltrami operator in $M$ and by $f^\prime= \dfrac{\partial}{\partial r}f$, for $f\in C^\infty (M)$. We assume that there exists a radial function $\varphi_R$ satisfying the following properties: 
\begin{itemize}
\item There exists a constant $\gamma>0$ such that
$\Delta \varphi_R (r)=\gamma$ for  $r\leq R$.
\item We have $\varphi_R'' (r) \leq 1$ for all $r\geq 0$.
\item There holds
$|\Delta \varphi_R (r) -\gamma| \leq C$ for every  $r\geq $0.
\item We have
$$ -F_2 - 16  \big((\Delta \varphi_R^\prime)' +(N-1)\frac{\phi'}{\phi} \Delta \varphi_R^\prime \big) +8(N-1)  \big(\frac{\phi'}{\phi}\big)' + \frac{3}{2} \big(F_1 ' + (N-1)\frac{\phi'}{\phi} F_1\big)  \leq o_R (1)$$
and
\begin{align*}
&8G -G_1 +\dfrac{1}{2} \big(\Delta F_2 +(N-1)\frac{\phi'}{\phi} F_2 ' +(N-1)^2 \big(\frac{\phi'}{\phi}\big)^2 F_2 + (N-1) \big(\frac{\phi'}{\phi}\big)' F_2\\
&\ \ \ \ \ - F_3 '- (N-1)\frac{\phi'}{\phi} F_3 \big)+\Delta^3 \varphi_R -\beta \Delta^2 \varphi_R \leq o_R (1).
\end{align*}
\item  $\varphi_R (r) \equiv C$ when $r\geq 10R$.
\end{itemize}
The functions $F_1,F_2,F_3,G$ and $G_1$ are defined in the following. We refer to \cite{bou-lenz} for a choice of admissible function $\varphi_R$ in the Euclidean setting and to the Maple file in \cite{maple} to produce examples of $\varphi_R$ in a non-flat case.

We define the localized virial of $u\in H^2(M)$ as
\begin{equation}\label{def-vir}
M_{\varphi_R} (u)=\big\langle u, -i (\langle\nabla \varphi_R, \nabla u\rangle + u\Delta \varphi_R)\big\rangle = 2 \im \int_{M} \bar{u} \langle\nabla \varphi_R, \nabla u\rangle\, dV.
\end{equation} 
The following time evolution inequality for $M_{\varphi_R}$ generalizes \cite[Lemma 3.1]{bou-lenz}.
\begin{lem}\label{locvirest}
Let $N\geq 2$ and $R>0$. Suppose that $u\in C([0,T);H^2 (M))$ is a radial solution to \eqref{4NLS}. Then, for any $t\in [0,T)$, we have
\begin{align*}
\dfrac{d}{dt}M_{\varphi_R}& \big(u(t)\big) \leq 4 \sigma \gamma E (u_0)+(8-2\sigma \gamma) \int_M |\Delta u|^2\, dV +\beta (4-2\sigma \gamma ) \int_M |\nabla u|^2\, dV \\
&+ 4\mu \int_{M} |\nabla u|^2 (\varphi_R'' -1)\,dV +o_R \big(1+\|\nabla u\|^2_{L^2(M)}\big)+O\big( R^{-\sigma (N-1)} \|\nabla u\|_{L^2 (M)}^{\sigma}\big). 
\end{align*}
\end{lem}
Above 
\[
 o_R \big( 1 + \|\nabla u\|_{L^2 (M)}^2\big) 
 \] 
 denotes quantities such that 
 \[
 \lim_{R\rightarrow \infty}\frac{ o_R \big( 1 + \|\nabla u\|_{L^2 (M)}^2\big) }{ 1 + \|\nabla u\|_{L^2 (M)}^2 } =0 .
 \]
\begin{proof}
We set $\Gamma_{\varphi_R}=-i \big(\langle\nabla \varphi_R, \nabla (\cdot)\rangle +\Delta \varphi_R\big)$. Taking the time derivative of \eqref{def-vir} and noticing that $i\partial_t u$ is given by \eqref{4NLS}, we get
$$\dfrac{d}{dt}M_{\varphi_R} \big(u(t)\big)=A_1\big(u(t)\big) +A_2\big(u(t)\big) +B\big(u(t)\big),$$
where
\[
A_1(u) = \big\langle u, [\Delta^2 , i\Gamma_{\varphi_R}] u\big\rangle,\quad
A_2 (u) = \big\langle u, [-\beta\Delta , i\Gamma_{\varphi_R}] u\big\rangle,
\]
and
\[
B(u) =  \big\langle u, [-|u|^p , i\Gamma_{\varphi_R}] u\big\rangle.
\]
First, we deal with $A_2$. After a long but straight-forward computation, we have that
\begin{align*}
& [\Delta , i\Gamma_{\varphi_R}] u\\
&=[\Delta  (\varphi_R ' u' +\Delta \varphi_R u) -\varphi_R ' (\Delta u)' - \Delta \varphi_R \Delta u]+\Delta \langle\nabla u, \nabla \varphi_R\rangle - \langle\nabla \varphi_R,\nabla \Delta u\rangle \\
&= 4 u'' \varphi_R '' + 4u' \Delta \varphi_R ' + u \Delta^2 \varphi_R.
\end{align*}
Integrating by parts, we deduce from the previous line that
\begin{align*}
A_2 (u)&=
-\beta \int_{M} u (4 u'' \varphi_R '' +4u' \Delta \varphi_R '  + u\Delta^2 \varphi_R)\, dV\\
&= 4\beta \int_{M} (u')^2 \varphi_R ''\, dV-\beta \int_{M} u^2  \Delta^2 \varphi_R\, dV\\
&=4\beta \int_{M} |\nabla u|^2\, dV+ 4\beta \int_M (u')^2 (\varphi_R'' -1 )\,dV -\beta \int_{M} u^2  \Delta^2 \varphi_R\, dV . 
\end{align*}
Next, we deal with $A_1$. We have
\begin{align*}
[\Delta^2 , i\Gamma_{\varphi_R}] u
&=\big[\Delta^2  (\varphi_R ' u' +\Delta \varphi_R u) -\varphi_R ' (\Delta^2 u)' - \Delta \varphi_R \Delta^2 u\big]+\Delta^2 (u' \varphi_R') -  \varphi_R' (\Delta^2 u)'\\
&=8 \varphi_R '' u^{(4)} + u^{(3)} \big(-2\varphi_R' A' +4\varphi_R'' A +4 \Delta \varphi_R ' +8 \varphi_R^{(3)} +4 (\Delta \varphi_R)' \big)\\
&\quad + u'' \Big(- 2\varphi_R' B' +4 \varphi_R '' (N-1) \big(\frac{\phi''}{\phi} - (\frac{\phi'}{\phi})^2 \big) +2A (\Delta \varphi_R ' + (\Delta \varphi_R)' ) \\
&\qquad + 4(\Delta \varphi_R')' +4\Delta \varphi_R '' +2 \Delta^2 \varphi_R +4 (\Delta \varphi_R)'' \Big)\\
&\quad +u' \Big(-2\varphi_R' C' +2(\Delta \varphi_R)' (N-1) \big(\frac{\phi'}{\phi}\big)'\\
&\qquad +2 \Delta^2 \varphi_R (N-1)\frac{\phi'}{\phi}+2\Delta^2 \varphi_R' +2 (\Delta^2 \varphi_R)'  +2 \Delta ( (\Delta \varphi_R)' ) \Big)\\
&\quad + u \Delta^3 \varphi_R ,
\end{align*}
where $$A= 2(N-1)\frac{\phi^\prime}{\phi},$$
$$B=2(N-1) \left(\frac{\phi''}{\phi}- \big(\frac{\phi'}{\phi}\big)^2 \right)+ (N-1)^2 \left(\frac{\phi^\prime}{\phi}\right)^2 , $$
and
$$C= (N-1) \left(\frac{\phi^{(3)}}{\phi}-3\frac{\phi'' \phi'}{\phi^2}+2 \big(\frac{\phi'}{\phi}\big)^3 \right)+ (N-1)^2 \left( \frac{\phi' \phi''}{\phi^2} - \big(\frac{\phi'}{\phi}\big)^3 \right) .$$
Let
\begin{align*}
F_1 &= -2\varphi_R' A' +4\varphi_R'' A +4 \Delta \varphi_R ' +8 \varphi_R^{(3)} +4 (\Delta \varphi_R)' ,\\
F_2 &= - 2\varphi_R' B' +4 \varphi_R '' (N-1) \big(\frac{\phi''}{\phi} - (\frac{\phi'}{\phi})^2 \big) +2A (\Delta \varphi_R ' + (\Delta \varphi_R)' )\\
&\quad + 4(\Delta \varphi_R')' +4\Delta \varphi_R '' +2 \Delta^2 \varphi_R +4 (\Delta \varphi_R)'' ,\\
\noalign{and}
F_3 &= -2\varphi_R' C' +2(\Delta \varphi_R)' (N-1) \big(\frac{\phi'}{\phi}\big)' +2 \Delta^2 \varphi_R (N-1)\frac{\phi'}{\phi}+2\Delta^2 \varphi_R' \\
&\quad +2 (\Delta^2 \varphi_R)'  +2 \Delta ( (\Delta \varphi_R)' ).
\end{align*}
Integrating by parts, we find that
\begin{align*}
&\int_M u u^{(4)} \varphi_R ''\, dV\\
&= \int_M \varphi_R ''( u'')^2\, dV - 2\int_M |u'|^2 \big((\Delta \varphi_R')' +(N-1)\frac{\phi'}{\phi} \Delta \varphi_R' \big)\, dV+ \int_M u^2 G\, dV,
\end{align*}
where
\begin{align*}
2G&= (\Delta^2 \varphi_R')' + 2(N-1) \frac{\phi'}{\phi}\Delta^2 \varphi_R' + (\Delta \varphi_R')' \big(2(N-1) \big(\frac{\phi'}{\phi}\big)' + (N-1)^2 \big(\frac{\phi'}{\phi}\big)^2 \big)\\
&\quad + \Delta \varphi_R' \big((N-1) \Delta \big(\frac{\phi'}{\phi}\big) +(N-1)^2 \big(\big(\frac{\phi'}{\phi}\big)^2\big)' + (N-1)^3 \big(\frac{\phi'}{\phi}\big)^3 \big).
\end{align*}
Doing the same for the other terms, we get
\begin{align*}
&\int_M u u^{(3)} F_1\, dV
=\frac{3}{2} \int_M (u')^2 \big(F_1 ' + (N-1)\frac{\phi'}{\phi} F_1\big)\,dV - \int_M u^2 G_1\, dV,
\end{align*}
where
\begin{align*}
2G_1&= (\Delta F_1)' + 2(N-1)\frac{\phi'}{\phi} \Delta F_1 +F_1 ' \big(2(N-1) \big(\frac{\phi'}{\phi}\big)' + (N-1)^2 \big(\frac{\phi'}{\phi}\big)^2 \big) \\
&+F_1 \big((N-1) \Delta \big(\frac{\phi'}{\phi}\big) + (N-1)^2  \big(\big(\frac{\phi'}{\phi}\big)^2\big)' + (N-1)^3 \big(\frac{\phi'}{\phi}\big)^3 \big),
\end{align*}
\begin{align*}
\int_M u u''& F_2\, dV
=   -\int_M (u')^2 F_2\, dV \\
& + \frac{1}{2} \int_M u^2 \big(\Delta F_2 +(N-1)\frac{\phi'}{\phi} F_2 ' +(N-1)^2 \big(\frac{\phi'}{\phi}\big)^2 F_2 + (N-1) \big(\frac{\phi'}{\phi}\big)' F_2 \big)\,dV,
\end{align*}
and
\begin{align*}
&\int_M u u' F_3\, dV
=  -  \frac{1}{2}\int_M u^2  \big(F_3 '+ (N-1)\frac{\phi'}{\phi} F_3 \big)\,dV.
\end{align*}
Also, noticing that
$$\int_M |\Delta u|^2\, dV= \int_M |u''|^2\, dV - (N-1) \int_M |u'|^2 \big(\frac{\phi'}{\phi}\big)'\, dV.$$
Thanks to our assumption on $\varphi_R$, we observe that
\begin{align*}
\int_M |\nabla u|^2 &\Big(-F_2 - 16  \big((\Delta \varphi_R')' +(N-1)\frac{\phi'}{\phi} \Delta \varphi_R' \big) +8(N-1)  \big(\frac{\phi'}{\phi}\big)' \\
&\qquad + \frac{3}{2} \big(F_1 ' + (N-1)\frac{\phi'}{\phi} F_1\big)  \Big)\,dV\\
&\leq o_R \big( \|\nabla u\|_{L^2 (M)}^2\big),
\end{align*}
and
\begin{align*}
\int_M |u|^2 &\Big(8G -G_1 +\dfrac{1}{2} \big(\Delta F_2 +(N-1)\frac{\phi'}{\phi} F_2 ' +(N-1)^2 \big(\frac{\phi'}{\phi}\big)^2 F_2\\
& \qquad + (N-1) \big(\frac{\phi'}{\phi}\big)' F_2- F_3 '- (N-1)\frac{\phi'}{\phi} F_3 \big)+\Delta^3 \varphi_R -\beta \Delta^2 \varphi_R\Big)\, dV\\
& \leq o_R (1).
\end{align*}
This implies that
\begin{align}
\label{viriale1}
A_1 (u) +A_2 (u) &\leq 8 \int_M |\Delta u|^2\, dV+4\beta \int_M |\nabla u|^2\, dV+4\beta \int_M |\nabla u|^2 \big(\varphi_R'' -1)\,dV\\
& +o_R \big( 1 + \|\nabla u\|_{L^2 (M)}^2\big).\nonumber
\end{align}
Finally, we consider the term $B(u)$. One can see that
$$B(u)= 2\int_{M} |u|^2 \varphi_R^\prime \big(|u|^{2\sigma}\big)^\prime\, dV =-\frac{2\sigma}{\sigma +1}\int_{M} \Delta \varphi_R |u|^{2\sigma +2}\, dV. $$
Using that $\Delta \varphi_R (r) = \gamma $ if $r\leq R$, we get
\begin{equation}
\label{viriale2}B(u)=-\frac{2\sigma \gamma}{\sigma +1}\int_M |u|^{2\sigma +2}\,dV +\frac{2\sigma}{\sigma +1}\int_{M} |\Delta \varphi_R -\gamma | |u|^{2\sigma +2}\, dV. 
\end{equation}
From $|\Delta \varphi_R -\gamma |\leq C $, using the Strauss inequality and the fact that $r^{N-1}\leq \phi^{N-1}(r)$, for all $r\geq 0$, we find 
 \begin{align*}
\int_{M} |\Delta \varphi -\gamma | |u|^{2\sigma +2}\, dV& \leq C\int_{M\setminus B_R} |u|^{2\sigma +2}\,dV \\
&\leq C \|u\|_{L^2 (M)}^2 R^{-\sigma (N-1)} \|\nabla u\|_{L^2,\text{Eucl}}^{\sigma}\\
& \leq C\|u\|_{L^2 (M)}^2 R^{-\sigma (N-1)} \|\nabla u\|_{L^2 (M)}^{\sigma}. 
\end{align*} 
The result then follows from \eqref{viriale1}, \eqref{viriale2} and the last estimate.
\end{proof}

\begin{thm} 
Let $M=(\R^N,g)$ be a complete rotationally symmetric manifold equipped with the Riemannian metric 
$g =dr^2 +\phi^2(r) d\theta^2$, with $\phi(r)\ge r$. Suppose, furthermore, that there exists the radial function $\varphi_R$ as defined previously for any $R>0$.
Let $N\geq 2$, $\beta \in \R\setminus\{0\}$, $\sigma \leq 4$, $\sigma \gamma >4$ and $\sigma < \dfrac{N}{(N-4)_+}$.  Suppose that $u_0 \in H^2 (M)$ is radial such that
$$E(u_0)<\begin{cases}0,&\mbox{ if}\ \beta > 0,\\ - a(N,\sigma) \beta^2 \|u_0\|_{L^2 (M)}^2,&\mbox{ if}\ \beta<0, \end{cases}$$
for some constant $a(N,\sigma)$ depending on $N$ and $\sigma$. 
Then, the solution $u\in C([0,T);H^2 (M))$ of \eqref{4NLS} blows-up in finite time.
\end{thm}

\begin{proof}
We first consider the case $\beta> 0$ and $E(u_0)<0$. From the previous Lemma \ref{locvirest}, we have
\begin{align*}
\dfrac{d}{dt}M_{\varphi_R} \big(u(t)\big)& \leq 4 \sigma \gamma E (u_0)+(8-2\sigma \gamma) \int_M |\Delta u|^2\, dV \\
& +o_R \big(1+\|\nabla u\|^2_{L^2(M)}\big)+O\big( R^{-\sigma (N-1)} \|\nabla u\|_{L^2 (M)}^{\sigma}\big). 
\end{align*}
Using that $\|\nabla u(t)\|_{L^2 (M)}\leq C(u_0) \|\Delta u (t)\|_{L^2 (M)}^{1/2} $ and $\sigma \leq 4$, we can choose $R>0$ large enough, such that
\begin{equation}
\label{blowe1}
\dfrac{d}{dt}M_{\varphi_R} (u(t)) \leq  -\delta \int_M |\Delta u|^2\, dV,\ t\in [0,T), 
\end{equation}
for some constant $\delta>0$. Suppose on the contrary  that $T=\infty$. From \eqref{blowe1}, we see that there exists $t_1 >0$ such that $M_{\varphi_R} \big(u(t)\big) \leq 0$ for all $t\geq t_1$. So, integrating \eqref{blowe1} over $[t_1 ,t]$, $t>t_1$, and using Cauchy-Schwartz inequality, we find
$$M_{\varphi_R} \big(u(t)\big) \leq -\delta \int_{t_1}^t \|\Delta u(s)\|_{L^2 (M)}^2\, ds\leq - C(\delta ,R) \int_{t_1}^t \big|M_{\varphi_R} \big(u(s)\big) \big|^4\, ds . $$
Setting 
\[
z(t)= \int_{t_1}^t \big|M_{\varphi_R} \big(u(s)\big) \big|^4\, ds,
\]
 we see that $z^\prime (t) \geq C z(t)^4$. It is easy to see that $z(t)$ has to blow-up in finite time. Therefore, $u(t)$ cannot exist for all $t\geq 0$.

Next, we consider the case $\beta<0$. In this case, we use that 
\[
\|\nabla u\|_{L^2 (M)}^2 \leq \dfrac{1}{2\eta} \|u\|_{L^2 (M)}^2 + \dfrac{\eta}{2} \|\Delta u\|_{L^2 (M)},
\]
 for some $\eta>0$. So, proceeding as above, we get
\begin{align*}
\dfrac{d}{dt}M_{\varphi_R} \big(u(t)\big)& \leq 4 \sigma \gamma E (u_0)+ \dfrac{A^2 \beta^2}{4 (N\sigma -4)} \int_M |u_0|^2\, dV -\delta \int_M |\Delta u|^2\, dV \\
& +o_R (1), 
\end{align*}
where $A=\beta \big(4- 2\sigma \gamma +4 \max_{r\geq 0} |\varphi'' (r) -1|\big) $ and $\delta >0$. So assuming that 
\[
4 \sigma \gamma E (u_0)+ \frac{A^2 \beta^2}{4 (N\sigma -4)} \int_M |u_0|^2 \,dV <0,
\]
 we get that
\begin{equation*}
\dfrac{d}{dt}M_{\varphi_R} \big(u(t)\big) \leq  -\delta \int_M |\Delta u|^2\, dV,\  t\in [0,T). 
\end{equation*}
At this point, we can conclude as previously.
\end{proof}


\begin{thebibliography}{10}

\bibitem{maple}
Maple file.
\newblock {\em https://www.dropbox.com/s/tmrv6kimhsgx82h/4NLS.mw?dl=0}.

\bibitem{AMPVZ}
Jean-Philippe Anker, Stefano Meda, Vittoria Pierfelice, Maria Vallarino, and
  Hong-Wei Zhang.
\newblock Schr\"odinger equation on noncompact symmetric spaces.
\newblock {\em Preprint arXiv:2104.00265 [math.AP]}, 2021.

\bibitem{MR2566713}
Jean-Philippe Anker and Vittoria Pierfelice.
\newblock Nonlinear {S}chr\"{o}dinger equation on real hyperbolic spaces.
\newblock {\em Ann. Inst. H. Poincar\'{e} Anal. Non Lin\'{e}aire},
  26(5):1853--1869, 2009.

\bibitem{MR2765426}
Jean-Philippe Anker, Vittoria Pierfelice, and Maria Vallarino.
\newblock Schr\"{o}dinger equations on {D}amek-{R}icci spaces.
\newblock {\em Comm. Partial Differential Equations}, 36(6):976--997, 2011.

\bibitem{banica2007}
Valeria Banica.
\newblock The nonlinear {S}chr\"{o}dinger equation on hyperbolic space.
\newblock {\em Comm. Partial Differential Equations}, 32(10-12):1643--1677,
  2007.

\bibitem{BaCaSt}
Valeria Banica, R\'{e}mi Carles, and Gigliola Staffilani.
\newblock Scattering theory for radial nonlinear {S}chr\"{o}dinger equations on
  hyperbolic space.
\newblock {\em Geom. Funct. Anal.}, 18(2):367--399, 2008.

\bibitem{BaDyweight}
Valeria Banica and Thomas Duyckaerts.
\newblock Weighted {S}trichartz estimates for radial {S}chr\"{o}dinger equation
  on noncompact manifolds.
\newblock {\em Dyn. Partial Differ. Equ.}, 4(4):335--359, 2007.

\bibitem{MR3338838}
Valeria Banica and Thomas Duyckaerts.
\newblock Global existence, scattering and blow-up for the focusing {NLS} on
  the hyperbolic space.
\newblock {\em Dyn. Partial Differ. Equ.}, 12(1):53--96, 2015.

\bibitem{MR1745182}
Matania Ben-Artzi, Herbert Koch, and Jean-Claude Saut.
\newblock Dispersion estimates for fourth order {S}chr\"{o}dinger equations.
\newblock {\em C. R. Acad. Sci. Paris S\'{e}r. I Math.}, 330(2):87--92, 2000.

\bibitem{BeGa}
Elvise Berchio and Debdip Ganguly.
\newblock Improved higher order {P}oincar\'{e} inequalities on the hyperbolic
  space via {H}ardy-type remainder terms.
\newblock {\em Commun. Pure Appl. Anal.}, 15(5):1871--1892, 2016.

\bibitem{MR3855391}
Denis Bonheure, Jean-Baptiste Casteras, Ederson~Moreira dos Santos, and Robson
  Nascimento.
\newblock Orbitally stable standing waves of a mixed dispersion nonlinear
  {S}chr\"{o}dinger equation.
\newblock {\em SIAM J. Math. Anal.}, 50(5):5027--5071, 2018.

\bibitem{MR4001029}
Denis Bonheure, Jean-Baptiste Cast\'{e}ras, Tianxiang Gou, and Louis Jeanjean.
\newblock Strong instability of ground states to a fourth order
  {S}chr\"{o}dinger equation.
\newblock {\em Int. Math. Res. Not. IMRN}, (17):5299--5315, 2019.

\bibitem{bou-lenz}
Thomas Boulenger and Enno Lenzmann.
\newblock Blowup for biharmonic {NLS}.
\newblock {\em Ann. Sci. \'{E}c. Norm. Sup\'{e}r. (4)}, 50(3):503--544, 2017.

\bibitem{MR1909648}
Nicolas Burq, Patrick G\'{e}rard, and Nikolay Tzvetkov.
\newblock An instability property of the nonlinear {S}chr\"{o}dinger equation
  on {$S^d$}.
\newblock {\em Math. Res. Lett.}, 9(2-3):323--335, 2002.

\bibitem{BGT}
Nicolas Burq, Patrick G\'{e}rard, and Nikolay Tzvetkov.
\newblock On nonlinear {S}chr\"{o}dinger equations in exterior domains.
\newblock {\em Ann. Inst. H. Poincar\'{e} Anal. Non Lin\'{e}aire},
  21(3):295--318, 2004.

\bibitem{MR2058384}
Nicolas Burq, Patrick G\'{e}rard, and Nikolay Tzvetkov.
\newblock Strichartz inequalities and the nonlinear {S}chr\"{o}dinger equation
  on compact manifolds.
\newblock {\em Amer. J. Math.}, 126(3):569--605, 2004.

\bibitem{BPST}
Nicolas Burq, Fabrice Planchon, John~G. Stalker, and A.~Shadi Tahvildar-Zadeh.
\newblock Strichartz estimates for the wave and {S}chr\"{o}dinger equations
  with potentials of critical decay.
\newblock {\em Indiana Univ. Math. J.}, 53(6):1665--1680, 2004.

\bibitem{cowling}
Michael Cowling.
\newblock Herz's ``principe de majoration'' and the {K}unze-{S}tein phenomenon.
\newblock In {\em Harmonic analysis and number theory ({M}ontreal, {PQ},
  1996)}, volume~21 of {\em CMS Conf. Proc.}, pages 73--88. Amer. Math. Soc.,
  Providence, RI, 1997.

\bibitem{CoMeSe}
Michael Cowling, Stefano Meda, and Alberto~G. Setti.
\newblock An overview of harmonic analysis on the group of isometries of a
  homogeneous tree.
\newblock {\em Exposition. Math.}, 16(5):385--423, 1998.

\bibitem{MR2296623}
Thomas Duyckaerts.
\newblock A singular critical potential for the {S}chr\"{o}dinger operator.
\newblock {\em Canad. Math. Bull.}, 50(1):35--47, 2007.

\bibitem{MR4069234}
Hongliang Feng, Avy Soffer, Zhao Wu, and Xiaohua Yao.
\newblock Decay estimates for higher-order elliptic operators.
\newblock {\em Trans. Amer. Math. Soc.}, 373(4):2805--2859, 2020.

\bibitem{MR1898529}
Gadi Fibich, Boaz Ilan, and George Papanicolaou.
\newblock Self-focusing with fourth-order dispersion.
\newblock {\em SIAM J. Appl. Math.}, 62(4):1437--1462, 2002.


\bibitem{MR1151250}
Jean Ginibre and Giorgio Velo.
\newblock    Smoothing properties and retarded estimates for some dispersive evolution equations.
\newblock {\em Comm. Math. Phys.}, 144 (1992), no. 1, 163-188. 

	

\bibitem{MR2211154}
Michael Goldberg, Luis Vega, and Nicola Visciglia.
\newblock Counterexamples of {S}trichartz inequalities for {S}chr\"{o}dinger
  equations with repulsive potentials.
\newblock {\em Int. Math. Res. Not.}, pages Art. ID 13927, 16, 2006.

\bibitem{helga}
Sigurdur Helgason.
\newblock Radon-{F}ourier transforms on symmetric spaces and related group
  representations.
\newblock {\em Bull. Amer. Math. Soc.}, 71:757--763, 1965.

\bibitem{I-JFA2000}
Alexandru~D. Ionescu.
\newblock Fourier integral operators on noncompact symmetric spaces of real
  rank one.
\newblock {\em J. Funct. Anal.}, 174(2):274--300, 2000.

\bibitem{IS2009}
Alexandru~D. Ionescu and Gigliola Staffilani.
\newblock Semilinear {S}chr\"{o}dinger flows on hyperbolic spaces: scattering
  {$H^1$}.
\newblock {\em Math. Ann.}, 345(1):133--158, 2009.

\bibitem{MR1779828}
V.~I. Karpman and A.~G. Shagalov.
\newblock Stability of solitons described by nonlinear {S}chr\"{o}dinger-type
  equations with higher-order dispersion.
\newblock {\em Phys. D}, 144(1-2):194--210, 2000.

\bibitem{Kato}
Tosio Kato.
\newblock Nonlinear {S}chr\" odinger equations. {S}chr\" odinger operators (Sonderborg, 1988), 218-263,
\newblock {\em Lecture Notes in Phys.}, 345, Springer, Berlin, 1989. 


\bibitem{keel-tao}
Markus Keel and Terence Tao.
\newblock Endpoint {S}trichartz estimates.
\newblock {\em Amer. J. Math.}, 120(5):955--980, 1998.

\bibitem{MR2502523}
Benoit Pausader.
\newblock The cubic fourth-order {S}chr\"{o}dinger equation.
\newblock {\em J. Funct. Anal.}, 256(8):2473--2517, 2009.

\bibitem{MR2505703}
Beno\^{\i}t Pausader.
\newblock The focusing energy-critical fourth-order {S}chr\"{o}dinger equation
  with radial data.
\newblock {\em Discrete Contin. Dyn. Syst.}, 24(4):1275--1292, 2009.

\bibitem{MR2245889}
Vittoria Pierfelice.
\newblock Weighted {S}trichartz estimates for the radial perturbed
  {S}chr\"{o}dinger equation on the hyperbolic space.
\newblock {\em Manuscripta Math.}, 120(4):377--389, 2006.

\end{thebibliography}
\end{document}